\documentclass[12pt]{article}
\usepackage[utf8]{inputenc}

\usepackage[colorlinks, citecolor=blue, linkcolor=blue]{hyperref}
\usepackage{amsmath, amsfonts, amssymb, amsthm, graphicx}
\usepackage[ruled]{algorithm2e}
\usepackage{subcaption}
\usepackage{orcidlink}
\usepackage{verbatim}

\usepackage{natbib}
\usepackage{setspace}
\newtheorem{proposition}{Proposition}
\newtheorem{theorem}{Theorem}
\newtheorem{lemma}{Lemma}

\newtheorem{assumption}{Assumption}

\newcommand{\norm}[1]{\left\lVert#1\right\rVert}

\newcommand{\e}{\epsilon}
\newcommand{\w}{\omega}
\renewcommand{\l}{\ell}
\newcommand{\Z}{\mathbb{Z}}

\newcommand{\E}{\mathbb{E}}
\newcommand{\Prob}{\mathbb{P}}
\renewcommand{\P}{\mathcal{P}}
\newcommand{\K}{\mathcal{K}}
\newcommand{\Var}{\text{Var}}
\newcommand{\F}{\mathcal{F}}

\newcommand{\X}{\textbf{X}}

\newcommand{\B}{\mathcal{B}}

\newcommand{\R}{\mathbb{R}}
\newcommand{\tr}{\text{tr}}

\let\temp\phi
\let\phi\varphi
\let\varphi\temp

\begin{document}

\title{A non-asymptotic error analysis for parallel Monte Carlo estimation from many short Markov chains}

\author{Austin Brown \orcidlink{0000-0003-1576-8381}, \\ Department of Statistical Sciences, University of Toronto, \\ ad.brown@utoronto.ca}
\maketitle

\begin{abstract}
Single-chain Markov chain Monte Carlo simulates realizations from a Markov chain to estimate expectations with the empirical average. The single-chain simulation is generally of considerable length and restricts many advantages of modern parallel computation. This paper constructs a novel many-short-chains Monte Carlo (MSC) estimator by averaging over multiple independent sums from Markov chains of a guaranteed short length. The computational advantage is the independent Markov chain simulations can be fast and may be run in parallel. The MSC estimator requires an importance sampling proposal and a drift condition on the Markov chain without requiring convergence analysis on the Markov chain. A non-asymptotic error analysis is developed for the MSC estimator under both geometric and multiplicative drift conditions. Empirical performance is illustrated on an autoregressive process and the Pólya-Gamma Gibbs sampler for Bayesian logistic regression to predict cardiovascular disease.
\end{abstract}



\section{Introduction}

Markov chain Monte Carlo (MCMC) is a widely applicable simulation technique to estimate a integrals with respect to a complex target probability distribution $\Pi$.
Accurate estimation of integrals is central to modern Bayesian inference where the target probability distribution may not be simulated from directly. 
In the \textit{single-chain} regime, MCMC simulates realizations from one Markov chain $X_0, \ldots, X_t$ until the marginal distribution is approximately from $\Pi$. 
After discarding realizations before this "burn-in" time, the Markov chain is simulated further and the empirical average is used to approximate integrals.

A persistent drawback of the single-chain regime is a necessarily long simulation length required to reduce the bias in the estimation.
It remains an active research area to develop explicit and practically useful convergence analyses on the mixing times of Markov chains.
This computational drawback has sparked research and debate into the \textit{many-short-chains} regime by instead combining simulations from multiple independent Markov chains of substantially shorter length \citep{gelman:rubin:1992, geyer:1992, margossian:2023}.
An advantage of the many-short-chains regime is modern parallel computation allows simulating thousands, even millions, of independent Markov chains at the same time.
In applications, the many-short-chains approach appears to empirically perform well for estimation in certain scenarios, but the shortened simulation length of the Markov chains may theoretically result in a large or even unknown bias.

Our main contribution constructs a novel many-short-chains Monte Carlo (MSC) estimator and develops a theoretical non-asymptotic error analysis for estimation.
The MSC estimator is the empirical average over multiple sums from independent Markov chains until their first return time to a particular set.
A significantly unique property of the MSC estimator is that it does not require a convergence analysis on the Markov chain. 
To the best of our knowledge, the non-asymptotic analysis developed is the first for a general purpose parallel MCMC estimator that does not require a mixing time analysis of the Markov chain.
From a solely theoretical viewpoint, the MSC estimator can provide a new perspective into the empirical performance of MCMC estimation in the many-short-chains regime.
The estimator may also be useful in non-trivial applications as it may be simulated with general Markov chains such as Gibbs samplers and is described in detail in Section~\ref{section:estimator} and applied in Section~\ref{section:pg_application}.

The non-asymptotic estimation error analysis ensures reliable performance of the MSC estimator but requires a proper and careful construction.
The estimator is built upon essential theory in Markov chains for the representation of invariant measures \citep[lemma 2]{nummelin:1976} and \citep[Theorem 10.4.9]{meyn:tweedie:2009}.
The independently generated Markov chains require a random initial distribution combined with a geometric or multiplicative drift condition.
The random initial distribution uses an importance sampling proposal for $\Pi$ that is constructed without the need to compute the normalizing constant of $\Pi$.
The length of the Markov chain simulation is determined by the time to return to a set $C$ and controlled by the drift condition.
Since the drift condition ensures the return time to the set $C$ is sufficiently fast, then the MSC estimator is guaranteed to remain in the many-short-chains regime.
Section~\ref{section:mse_error_bound} develops a non-asymptotic upper bound on the mean squared error under a geometric drift condition and when the importance weights have finite variance.
Section~\ref{section:concentration_mult} develops a concentration inequality under a stronger multiplicative drift condition and uniformly bounded importance weights.
In Section~\ref{section:toy_example} and Section~\ref{section:pg_application}, practical applications are illustrated on a simple autoregressive process and a non-trivial example with the Pólya-Gamma Gibbs sampler to predict cardiovascular disease.

The novel MSC estimator contributes to the recent research in parallel implementations of Markov chain Monte Carlo (\citep{craiu:meng:2005, craiu:etal:2009, lao:2020, pierre:2020} among others).
Unbiased Markov chain Monte Carlo is a recent approach to parallel estimation which constructs an unbiased estimator from multiple Markov chains \citep{glynn:rhee:2014, pierre:etal:2020}.
This requires constructing and simulating from a joint Markov kernel between a Markov chain and lagged version of itself until the first stopping time when the two chains exactly meet.
Constructing optimal joint Markov kernels is an active area of research \citep{wang:etal:2021} and obtaining explicit convergence bounds on the time when two Markov chains meet, even in moderate dimensions, can be challenging.
In particular, the MSC estimator may be beneficial in certain applications where Unbiased Markov chain Monte Carlo may not readily accessible.
For example, the length of the Markov chain simulations in the MSC estimator depend only on the drift condition which can be simpler to establish and scale efficiently to high dimensions.
However, this may not always be the case as explicit drift conditions are often available for Markov chains such as Gibbs samplers but can also be rare for others such as Metropolis-Hastings.
Section~\ref{section:conclusion} discusses benefits and drawbacks of our results and future research possibilities.

\section{The MSC estimator}
\label{section:estimator}

Let $\X$ be a Borel measurable metric space with Borel sigma field $\B(\X)$.
We will inherently assume all functions and sets are Borel measurable unless otherwise stated.
Let $\Pi$ be a target Borel probability measure on $\X$ and our goal is to estimate integrals through simulation of a suitable function $\phi : \X \to \R$, that is, to estimate
\[
\int_\X \phi d\Pi.
\]
A major motivation is estimating integrals with modern computation to perform Bayesian inference.

The MSC estimator uses independent sample paths from a time-homogeneous Markov chain $(X_t)_{t = 0}^{\infty}$ initialized from a random initial distribution constructed using importance sampling.
To construct the distribution initializing the Markov chains, let $Q$ be a Borel probability measure on $\X$, which must be chosen beforehand, used as an importance sampling proposal.
Define the importance weight $w = d\Pi/dQ$ which is the Radon-Nikodym derivative between $\Pi$ and $Q$.
Denote $\Z_+$ as the strictly positive integers.
For $N \in \Z_+$, let $Y = (Y_1, \ldots, Y_N)^T$ where $Y_j \sim Q$ are independent for $j = 1, \ldots, N$ and define the self-normalized importance weights
\[
w^N_Y(\cdot)
= \frac{ w(\cdot) }{ \sum_{j = 1}^N w(Y_j) }.
\]
For a point $x \in \X$, denote the Dirac probability measure by $\delta_x$.
Define the random initial probability distribution
\[
\Pi^N_Y
= \sum_{i = 1}^N w^N_Y(Y_i) \delta_{Y_i}.
\]
Since $\sum_{i = 1}^N w^N_Y(Y_i) = 1$, this defines a valid probability measure.

For each $x \in \X$, the conditional distribution of $X_t | X_{t - 1} = x$ is defined by a Markov transition kernel $\P$ on $\X \times \B(\X)$ for $t \in \Z_+$ with $\P^1 = \P$ and
\[
\P^t(x, dy) = \int_\X \P(z, dy) \P^{t-1}(x, dz).
\]
By this construction, we will assume $X_t | X_0 = x$ is independent of $Y$ since $\P$ is independent of $Y$.
We will also assume throughout that the Markov kernel $\P$ has a unique invariant measure $\Pi$, that is,
\begin{align}
\Pi \P = \Pi
\label{eq:unique_invariant}
\end{align}
is uniquely satisfied by $\Pi$.
In particular, if $X_0 \sim \Pi$, then $X_t \sim \Pi$ for all $t$ thereafter.
This assumption is satisfied for many useful Markov chains and verifiable conditions to ensure uniqueness for irreducible chains exist \citep[Theorem 10.4.4]{meyn:tweedie:2009}. For example, Metropolis-Hastings and Gibbs samplers often satisfy this assumption and is not restrictive for many practical applications.

We can then define a Markov chain $(X_t)_{t = 0}^{\infty}$ with a random distribution, in the sense of $Y$ being random, using initialization at $X_0 \sim \Pi^N_Y$ with conditional $X_t | X_{t - 1} = x \sim \P(x, \cdot)$ and marginal $X_t \sim \Pi^N_Y \P^t$ defined by the random probability measure
\[
\Pi^N_Y \P^t(dx') = \int_\X \P^t(x, dx') \Pi^N_Y(dx).
\]
With this Markov chain, we sum the sample path of the Markov chain until its first return to a set $C$.
For a set $C \subseteq \X$, define the first return time by $\tau_C = \inf\{ t \ge 1 : X_t \in C \}$.
For functions $\phi : \X \to \R$, define the sum
\begin{align*}
S_{\tau_C}(\phi) 
= \begin{cases}
\sum_{k = 1}^{\tau_C} \phi(X_k) & X_0 \in C \\
0 & X_0 \not\in C.
\end{cases}
\end{align*}

Let $M$ denote the the number of independent Markov chains and conditioned on $Y$, let $S^m_{\tau_C}(\phi) | Y$ be independent for $m = 1, \ldots, M$.
The \textit{MSC estimator} is the empirical average over these independent sums of Markov chains, that is, 
\[
\frac{1}{M} \sum_{m = 1}^M S^m_{\tau_C} (\phi).
\]
The algorithm used to simulate the MSC estimator is outlined below in Algorithm~\eqref{algorithm:pmcmc}.

\vspace{.5cm}
\begin{algorithm}[ht]
\caption{MSC estimator}
Simulate independent $Y_i \sim Q$ for $i = 1, \ldots, N$ to construct $\Pi_Y^N$

\For{$m = 1, \ldots, M$}{ 
  Simulate an independent Markov chain path $X^m_0, X_1^m, \ldots X_{\tau_C}^m$ and compute $S^m_{\tau_C}(\phi)$
}
Return $\frac{1}{M} \sum_{m = 1}^M S^m_{\tau_C} (\phi)$.
\label{algorithm:pmcmc}
\end{algorithm}
\vspace{.5cm}

\section{Mean squared error analysis under geometric drift conditions}
\label{section:mse_error_bound}

The main result of this section is developing a non-asymptotic analysis on the mean squared error in Theorem~\ref{theorem:mse_bound} if the importance weights have finite variance and under a geometric drift condition on the Markov chain.
However, some results in this section do not necessarily require the drift to be geometric and we define the drift condition more generally.
We will assume the following general drift condition (see \citep[Chapter 11]{meyn:tweedie:2009}) to control return times of the Markov chain to a specific set.

\begin{assumption}
\label{assumption:drift}
(Drift condition)
Suppose there are functions $V : \X \to [1, \infty)$, $f : \X \to [1, \infty)$, and constants $\gamma \in (0, 1)$ and $K \in (0, \infty)$ such that for each $x \in \X$
\[
(\P V)(x) - V(x)
\le - (1 - \gamma) f(x) + K.
\]
\end{assumption}

Under the drift condition \eqref{assumption:drift}, define the set $C$ for any $R > K / (1 - \gamma)$ by
\begin{align}
C = \{x \in \X : f(x) \le R \}.
\label{eq:set_C}
\end{align}
The drift condition \eqref{assumption:drift} controls the return times to $C$ and in particular, the drift condition \eqref{assumption:drift} is geometric if $f = V$.
In this case, the return time to a large enough sublevel set of the function $V$ can be rapid resulting in a short simulation of the Markov chain paths in Algorithm~\eqref{algorithm:pmcmc}.
We first introduce a conditional measure $\K_C$ for every set $B \subseteq \X$ and every $x \in \X$ by
\begin{align*}
\K_C(x, B)
= \E\left[ S_{\tau_C}(I_B) I_C(x) \bigm| X_0 = x \right].
\end{align*}
Under Theorem~\ref{theorem:representation}, we bound the mean squared bias for possibly unbounded functions dominated by $f$ defined in the drift condition \eqref{assumption:drift}. 
The result is inspired by an importance sampling result \citep[Theorem 2.1]{agapiou:2017}.

\begin{proposition}
\label{proposition:bias}
Assume the drift condition \eqref{assumption:drift} holds, and the set $C$ is defined by \eqref{eq:set_C}.
If $Y_i \sim Q$ for $i = 1, \ldots, N$ are independent and $\sup_{x \in C} V(x) < \infty$,
\[
\sup_{|\phi| \le f}\E\left| 
\int_\X \K_C \phi d\Pi_Y^N - \int_\X \phi d\Pi \right|^2
\le 
\frac{4 A_R^2 \int w d\Pi}{N(1 - \gamma_R )^2} - \frac{2 A_R^2}{N(1 - \gamma_R )^2}
\]
where $\gamma_R = \gamma + K/R$ and
$
A_R = \sup_{x \in C} \left[ V(x) + (\gamma - 1) f(x) \right] + 2 K - 1.
$
\end{proposition}

An interesting property of the mean squared bias in Proposition~\ref{proposition:bias} is that it holds under relatively mild conditions such as subgeometric drift conditions.
We can also see limitations of the upper bound if $V$ is unbounded on $C$.
We now bound the variance conditioned on the importance samples $Y$.
Here we will require a geometric drift condition \eqref{assumption:drift} where $f = V$ and a slightly more restrictive class of functions $|\phi| \le \sqrt{V}$.

\begin{proposition}
\label{proposition:variance}
Assume the geometric drift condition \eqref{assumption:drift} with $f = V$ holds and the set $C$ is defined by \eqref{eq:set_C}.
If $S^m_{\tau_C}(\phi) | Y$ are independent for $m = 1, \ldots, M$, then
\[
\sup_{|\phi| \le \sqrt{V}} \E\left[ \left| \frac{1}{M} \sum_{m = 1}^M S^m_{\tau_C}(\phi) - \int \K_C \phi d\Pi^N_Y \right|^2 \right]
\le \frac{[R + K]}{M} \left[ \frac{ \gamma_R/2 }{ 1 - \gamma_R/2 } \right]^2.
\]
\end{proposition}

Intuitively, the upper bound on the variance will increase if the radius $R$ of the set $C$ increases, which is a sublevel set of $V$.
We may now combine the results on the mean squared bias and variance to prove an upper bound on the mean squared error.

\begin{theorem}
\label{theorem:mse_bound}
Assume the a geometric drift condition \eqref{assumption:drift} with $f = V$ holds, and the set $C$ is defined by \eqref{eq:set_C}.
If $Y_i$ are independent for $i = 1, \ldots, N$ and $S^m_{\tau_C}(\phi) | Y$ are independent for $m = 1, \ldots, M$, then
\begin{align*}
\sup_{|\phi| \le \sqrt{V}} \E\left[ 
\left| 
\frac{1}{M} \sum_{m = 1}^M S^m_{\tau_C}(\phi) - \int_\X \phi d\Pi 
\right|^2
\right]
&\le \left(
\frac{\gamma_R \sqrt{R + K}}{ \sqrt{M} (2 - \gamma_R) }
+ \frac{ 2 A_R \sqrt{\int w d\Pi} }{\sqrt{ N } (1 - \gamma_R )}
\right)^2.
\end{align*}
\end{theorem}

The upper bound in Theorem~\ref{theorem:mse_bound} is a simple combination of Proposition~\ref{proposition:bias} and \ref{proposition:variance}.
Interestingly, the mean squared error bound holds for any sublevel set of $V$ defined in \eqref{eq:set_C} so long as $R$ is large enough.
To enforce shorter Markov chain path lengths, and hence short simulation lengths, one may choose a large set $C$.
However, this upper bound illustrates the effect of choosing a larger set $C$ as the number of independent chains will then need to increase.

An application of Theorem~\ref{theorem:mse_bound} and Markov's inequality yields an error analysis that can be useful for assessing the reliability of estimation in Bayesian inference and is demonstrated in applications in Section~\ref{section:pg_application}.
Let $\norm{\cdot}_p$ for $p \ge 1$ denote the standard p-norms.
Roughly speaking, the main application in mind is in Euclidean space $\R^d$ where $\norm{x}_2 \le \sqrt{V(x)}$.
Then if the variance of the importance weights is well-behaved, we have an explicit requirement on $N$ and $M$ so it is guaranteed with high probability, for any $|\phi(x)| \le \norm{x}_2$, the MSC estimator
\[
\frac{1}{M} \sum_{m = 1}^M S^m_{\tau_C}(\phi)
\approx \int_{\R^d} \phi d\Pi.
\]
This is of practical importance in Bayesian estimation where $\int_{\R^d} x \Pi(dx)$ corresponds to the posterior mean and the error analysis developed here may be applied coordinate-wise.

Alternative upper bounds to Theorem~\ref{theorem:mse_bound} can be developed as follows.
If the ratio satisfies $N / M \ge [\gamma R + 2K] \int w d\Pi$, then we can obtain the following bound by optimizing $\alpha$ below to get
\begin{align*}
\E\left[ \left| \frac{1}{M} \sum_{m = 1}^{M} S_{\tau_C}^m(\phi)
- \int_\X \phi d\Pi \right|^2 \right]
&\le (1 + \alpha) \frac{ [R + K] (\gamma_R / 2)^2 }{M (1 - \gamma_R/2)^2 }
+ \frac{
(1 + \alpha^{-1}) 4 [\gamma R + 2K] 
}{
M (1 - \gamma_R)^2 
}
\\
&\le 
\frac{ \gamma_R R + 2 K}{M (1 - \gamma_R)^2 }
\left\{ 
(1 + \alpha) \frac{1}{4}
+
(1 + \alpha^{-1}) 4 
\right\}
\\
&\le \frac{6.25 \gamma_R R + 12.5 K}{M (1 - \gamma_R)^2 }.
\end{align*} 
Here the scaling of $M$ is more clear with respect to $R, K, \gamma$, but the ratio $N/M$ plays a part in this simplified estimate. However, this can be estimated with importance sampling in practice to ensure such a requirement holds.

\section{Concentration results under multiplicative drift conditions}
\label{section:concentration_mult}

The goal of this section is to improve the error analysis in the previous section but introducing stronger conditions.
We look at a stronger multiplicative drift condition \eqref{assumption:drift} that is used for large deviation theory in Markov chains \citep{kontoyiannis:meyn:2005} and has received attention recently \citep{devraj:etal:2020}.   

\begin{assumption}
\label{assumption:drift_mult}
(Multiplicative drift condition)
Let $V : \X \to [0, \infty)$ and $f : \X \to [0, \infty)$.
Suppose for some $\gamma \in (0, 1)$ and $K \in [0, \infty)$ that for each $x \in \X$,
\[
\log\{ [ \P( \exp(V) ] (x) \}
\le V(x) - (1 - \gamma) f(x) + K.
\]
\end{assumption}

We first have a sub-Gaussian \citep{hoeffding:1963} concentration bound for the bias when the importance weights are uniformly bounded.
In many applications, assuming the importance weights are bounded is unreasonable, but if this is condition holds, the error analysis provided in this section can produce sharper results.

\begin{proposition}
\label{proposition:bias_bounded}
Assume the multiplicative drift condition \eqref{assumption:drift_mult} holds, and the set $C$ is defined by \eqref{eq:set_C}.
Suppose $\sup_{x \in C} V(x) < \infty$ and suppose $\sup_{x \in \X} w(x)  = w_* < \infty$.
If $Y_i \sim Q$ for $i = 1, \ldots, N$ are independent, then for any $\e \in (0, 1)$,
\[
\sup_{|\phi| \le f}\Prob\left( 
\left| 
\int_\X \K_C \phi d\Pi^N_Y
- \int_\X \phi d\Pi
\right|
\ge \e
\right)
\le 
4 \exp\left( 
-\frac{N \e^2 (1 - \gamma_R)^2}{2 \exp[2 B_R] w_*^2 }
\right)
\]
where $\gamma_R = \gamma + K/R$
and $B_R = \sup_{x \in C} \left[ V(x) - (1 - \gamma)f(x) \right] + 2K$.
\end{proposition}

Next, we obtain a sub-Gaussian concentration inequality under the multiplicative drift condition. 

\begin{proposition}
\label{proposition:chernoff}
Assume the multiplicative drift condition \eqref{assumption:drift_mult} holds and the set $C$ is defined by \eqref{eq:set_C}.
Assume $\sup_{x \in C} V(x) < \infty$.
If $S^m_{\tau_C}(\phi) | Y$ are independent for $m = 1, \ldots, M$, then for any $\e \in (0, 1)$,
\[
\sup_{|\phi| \le f} \Prob\left[ \left| \frac{1}{M} \sum_{m = 1}^M S^m_{\tau_C}(\phi) - \int \K_C \phi d\Pi^N_Y \right| \ge \e \right]
\le 2 \exp\left[ -\frac{M \e^2 (1 - \gamma_R)^2}{ 9 \exp(2 B_R))}
\right].
\]
\end{proposition}

We now have the main result of this section developing a concentration inequality in Theorem~\ref{theorem:mult_error_bound} under the stronger multiplicative drift condition and uniformly bounded importance weights.

\begin{theorem}
\label{theorem:mult_error_bound}
Assume the drift condition \eqref{assumption:drift_mult} holds, and the set $C$ is defined by \eqref{eq:set_C}. 
Suppose $\sup_{x \in C} V(x) < \infty$ and $\sup_{x \in \X} w(x)  = w_* < \infty$.
If $Y_i$ are independent for $i = 1, \ldots, N$ and $S^m_{\tau_C}(\phi) | Y$ are independent for $m = 1, \ldots, M$,
then for any $\e \in (0, 1)$, 
\begin{align*}
\sup_{|\phi| \le f} \Prob\left( \left| \frac{1}{M} \sum_{m = 1}^{M} S_{\tau_C}^m(\phi)
- \int_\X \phi d\Pi \right| \ge \e \right)
\le 6 \exp\left( 
-\frac{\e^2 (1 - \gamma_R)^2 \min\left\{ 2 M / 9, N / w_*^2 \right\} }{ 8 \exp(2 B_R) }
\right).
\end{align*}
\end{theorem}

\section{Toy example: Autoregressive process}
\label{section:toy_example}

Consider the simple Markov chain on $\R^d$ defined for $\rho \in (0, 1)$ and independent $\xi_t \sim N(0, I_d)$ for $t \in \Z_+$ by
\begin{align}
X_{t}
= \rho X_{t - 1} + \sqrt{1 - \rho^2} \xi_t.
\label{eq:ar_process}
\end{align}
Denote the normal distribution by $N(\mu, \Sigma)$ with mean $\mu$ and covariance matrix $\Sigma$.
With $\Gamma \equiv N(0, I_d)$, it is readily seen that if $X_0 \sim \Gamma$, then $X_t \sim \Gamma$.
We are interested in estimating the mean $\int x \Gamma(dx) \equiv 0$ through simulation.

We will choose importance sampling proposal $\Gamma_h \equiv N(0, (1/2 + h) I_d)$ with $h > 0$ to sample independently $Y_j \sim \Gamma_h$ for $j = 1,\ldots, N $ and construct the random initial distribution $\Pi^N_Y$ for the autoregressive process $X_t$ defined by \eqref{eq:ar_process}.
Using this chosen proposal, we can directly compute
\begin{align*}
\int w d\Gamma
&= \frac{[1/2 + h]^{d/2}}{(2 \pi)^{d/2}} \int 
\exp\left[
- \frac{1}{2} \left( \frac{2h}{1/2 + h} \right) \norm{x}_2^2 
\right] dx
= \frac{[1/2 + h]^{d}}{(2h)^{d/2}} 
\\
&= \left( \frac{1}{ 2 \sqrt{2 h} } + \frac{ \sqrt{h} }{ \sqrt{2} } \right)^{d}.
\end{align*}
The optimal $h = 1/2$ results in the invariant distribution for the proposal.

Using the identity for the variance, we have the geometric drift condition \eqref{assumption:drift}. Indeed for any $x \in \R^d$,
\begin{align*}
\E\left[ 1 +  \norm{X_1}_2^2 \bigm| X_0 = x\right]
&= \rho^2 ( 1 + \norm{x}_2^2 ) + (1 - \rho^2) (1 + d).
\end{align*}
Since the level sets of the drift function are compact, it can be shown the Markov chain converges geometrically fast and the invariant measure is unique \citep{hairer:mattingly:2011}.
For any $r > 1$, define the set $C_r = \{ x \in \R^d :  \norm{x}_2^2 \le r d \}$ and let $S^m_{ \tau_{C_r} }| Y$ independent $m = 1, \ldots, N$ using the autoregressive process \eqref{eq:ar_process}.

The drift controls the return times to the set $C_r$ and combining these two results, we can apply Theorem~\ref{theorem:mse_bound} to estimate functions satisfying $|\phi(x)| \le \sqrt{1 + \norm{x}_2^2}$.
Indeed, Theorem~\ref{theorem:mse_bound} says that if
\begin{align}
&N
\propto \frac{r^2 d^2}{\delta \e^2 (1 - \rho_r)^2 }
\left( \frac{1}{ 2 \sqrt{2 h} } + \sqrt{\frac{ h }{ 2 } } \right)^{d}
&M \propto \frac{ r d }{ \delta \e^2 ( 2 - \rho_r )^2}
\label{eq:toy_lb}
\end{align}
where $\rho_r = \rho^2 + (1 - \rho^2) (d + 1) / (rd + 1)$,
then  with probability at least $1 - \delta$, 
\[
\left| \frac{1}{M} \sum_{m = 1}^M S^m_{\tau_{C_r}}(\phi)
- \int_{\R^d} \phi d\Gamma \right| 
\le \e.
\]

Figure~\ref{figure:ar_bound} plots the precise theoretical lower bounds \eqref{eq:toy_lb} on the number of samples for $N$ and $M$ to increasing dimensions $\{ 1, 5, 10, 15, 20, 25, 30 \}$ with tuning parameters $\e = \delta = .1$, $\rho = .9$, $h = .49$, and $r = 1.5$.
We can see a linear scaling with the dimension for the number of Markov chains $M$.
We also see roughly quadratic scaling with respect to the dimension for $N$ although the lower bound on the required number of samples $N$ in the initialization is much larger.
However, alternative tuning parameter choices on the proposal may lead to approximately exponential scaling with respect to the dimension.

\begin{figure}[t]
\centering
\begin{subfigure}{.4\linewidth}
\includegraphics[width=\linewidth]{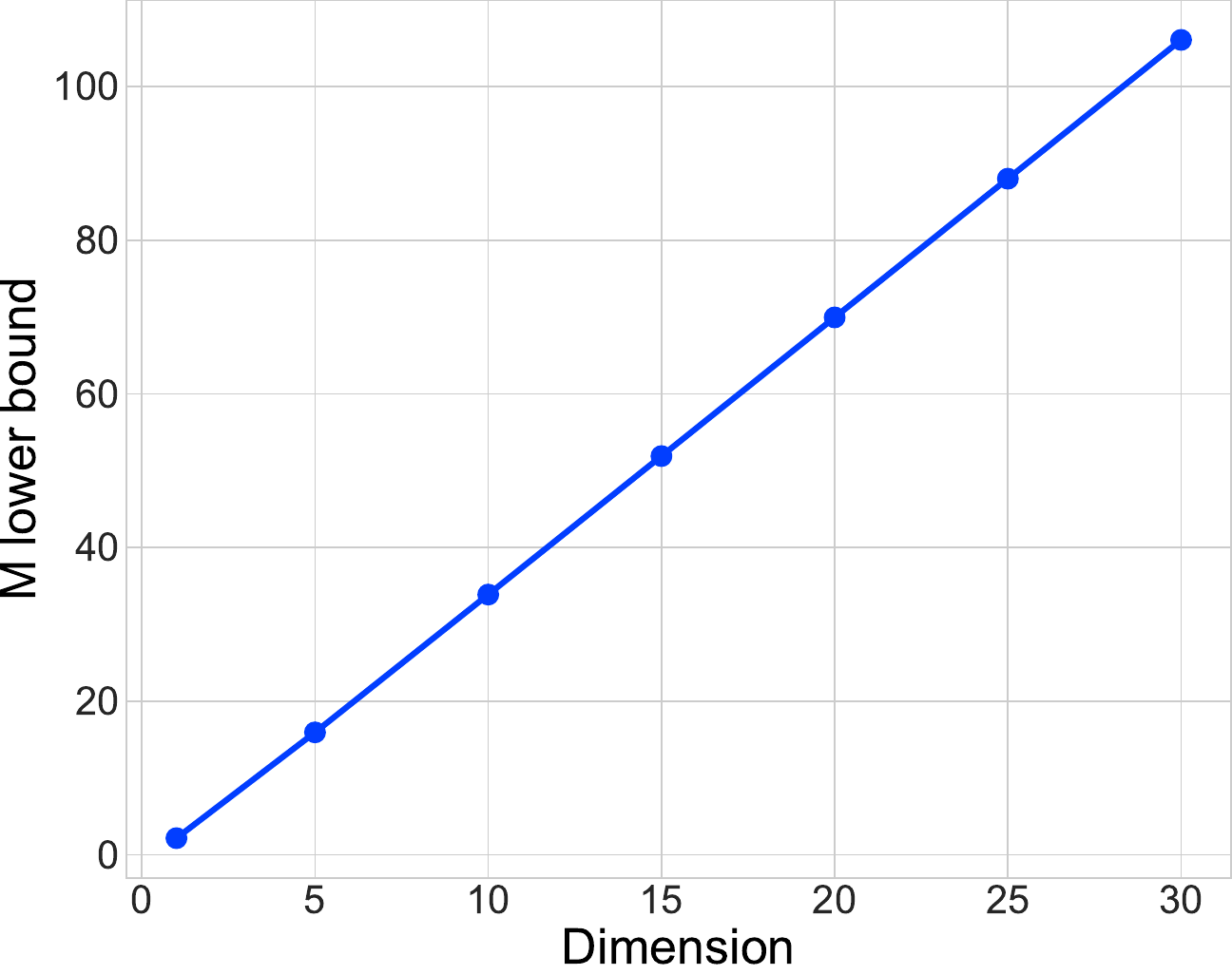}
\end{subfigure}
\hspace{.1em}
\begin{subfigure}{.44\linewidth}
\includegraphics[width=\linewidth]{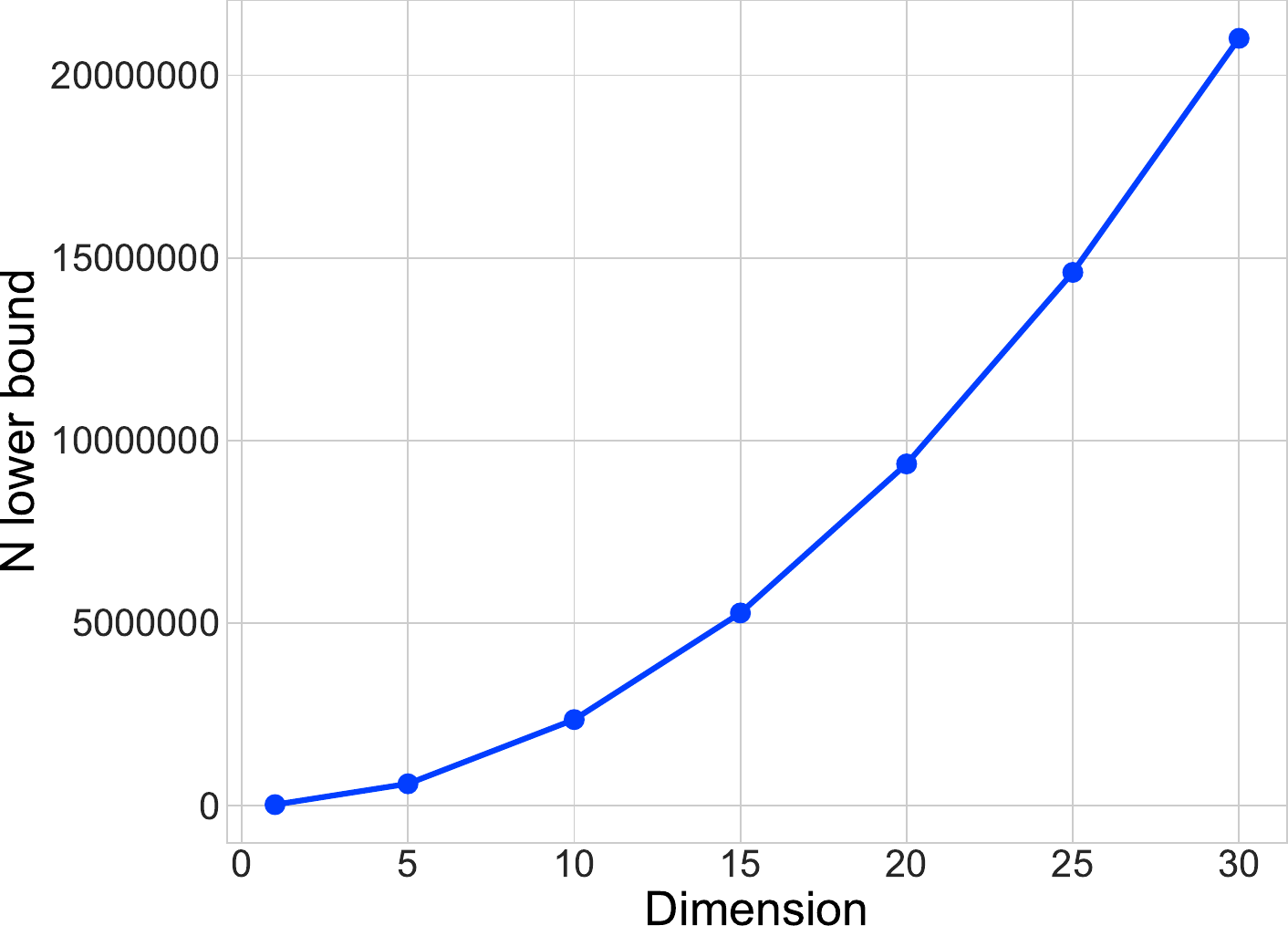}
\end{subfigure}
\caption{
Computation of the lower bound on the number of Markov chains $M$ and initialization samples $N$ for the MSC estimator using the autoregressive process with respect to increasing dimensions.
} 
\label{figure:ar_bound}
\end{figure}

We look to investigate the empirical estimation performance of the MSC estimator using this autoregressive process.
We attempt to estimate coordinates of the mean so that $\phi(x) = x_k$ for $k \le d$ via an MSC simulation.
Use the tuning parameters introduced previously, we simulate the MSC estimate with $N = 10^6$ and $M = 10^5$ independent samples in dimension $d = 2$.
In this case, we can directly sample the target distribution and compare to standard Monte Carlo estimation.

A comparison is made to the MSC simulation estimate within 2 estimated standard errors to standard Monte Carlo realizations and its estimated standard error.
Figures \ref{figure:ar_sim_a} and \ref{figure:ar_sim_b} illustrate the results of the simulation by plotting both coordinates individually.
Compared to standard Monte Carlo, the MSC estimate shows increased variability but can be utlized when standard Monte Carlo is not available.
Figure~\ref{figure:ar_sim_c} shows the run length of the independent Markov chain simulations until each Markov chain returns to the set $C_r$ defined in this section.
Here we observe an alignment with the theoretical analysis that the Markov chain has a rapid return time and remains in the many-short-chains regime.
In simulations not presented here, we observed it is possible to increase the radius of the set $C_r$ to further reduce the run length while also observing stable estimation performance. 

\begin{figure}[t]
\centering
\begin{subfigure}{.32\linewidth}
  \includegraphics[width=\linewidth]{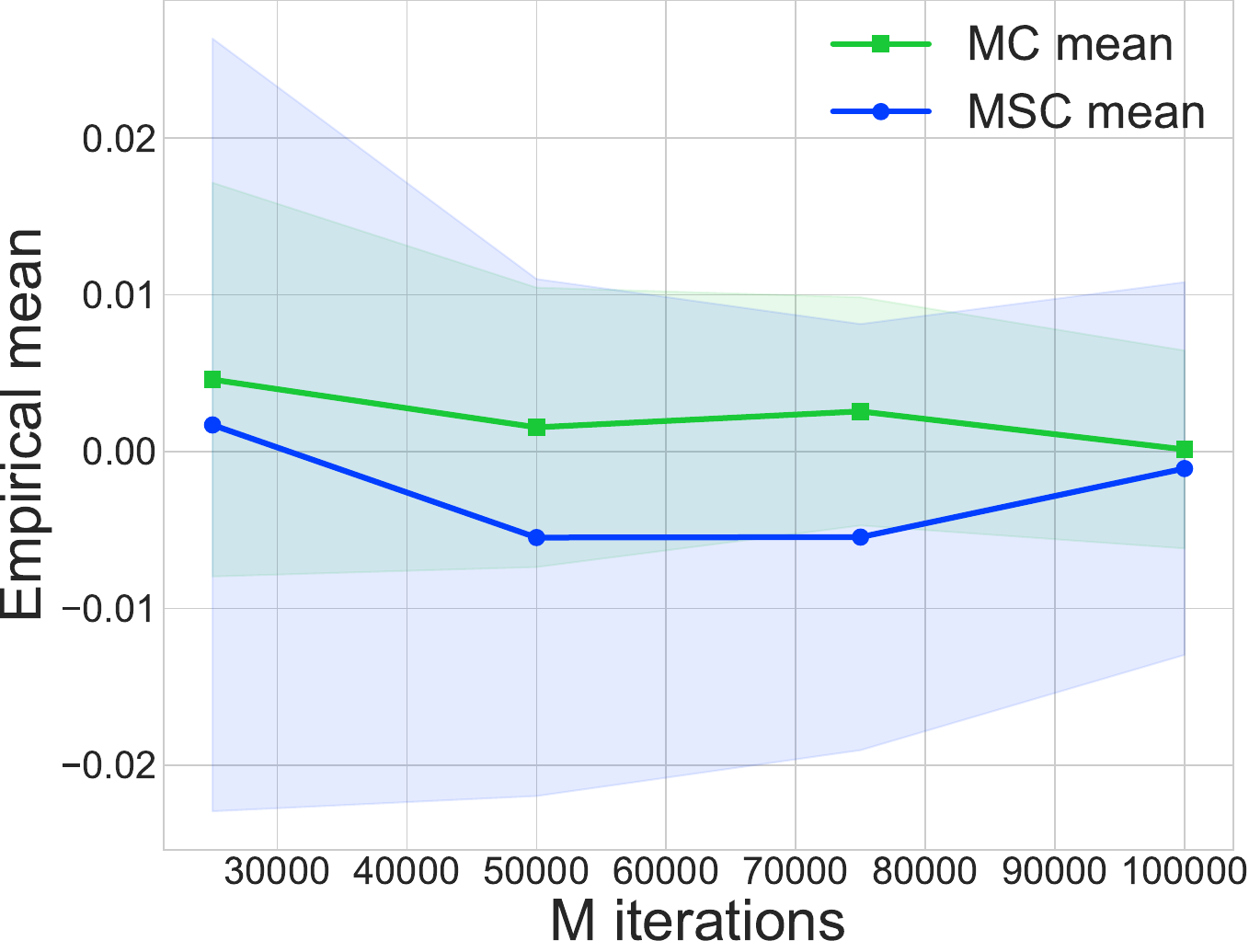}
  \caption{}
  \label{figure:ar_sim_a}
\end{subfigure}
\hspace{.2cm}
\begin{subfigure}{.32\linewidth}
  \includegraphics[width=\linewidth]{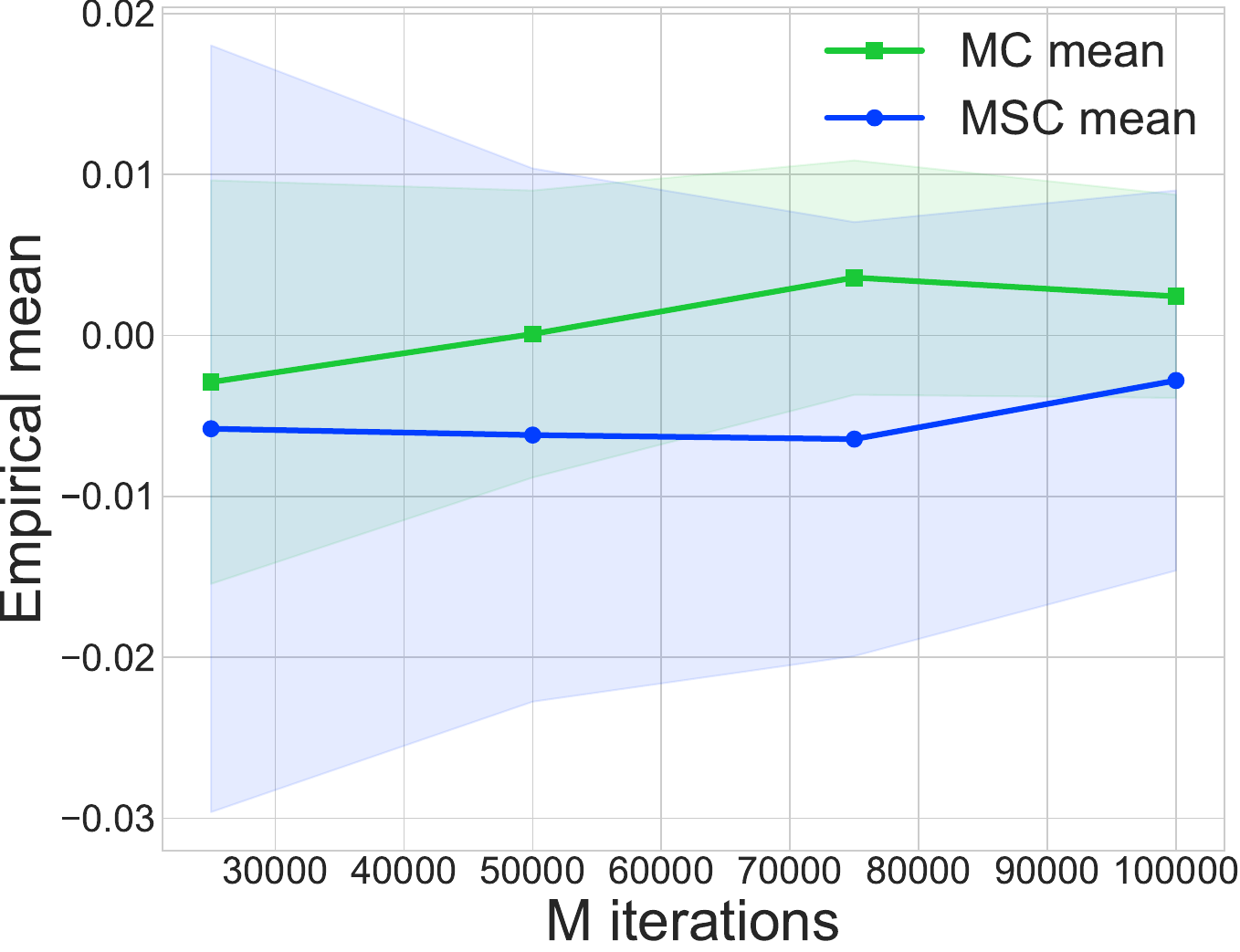}
  \caption{}
  \label{figure:ar_sim_b}
\end{subfigure}
\begin{subfigure}{.32\linewidth}
  \includegraphics[width=\linewidth]{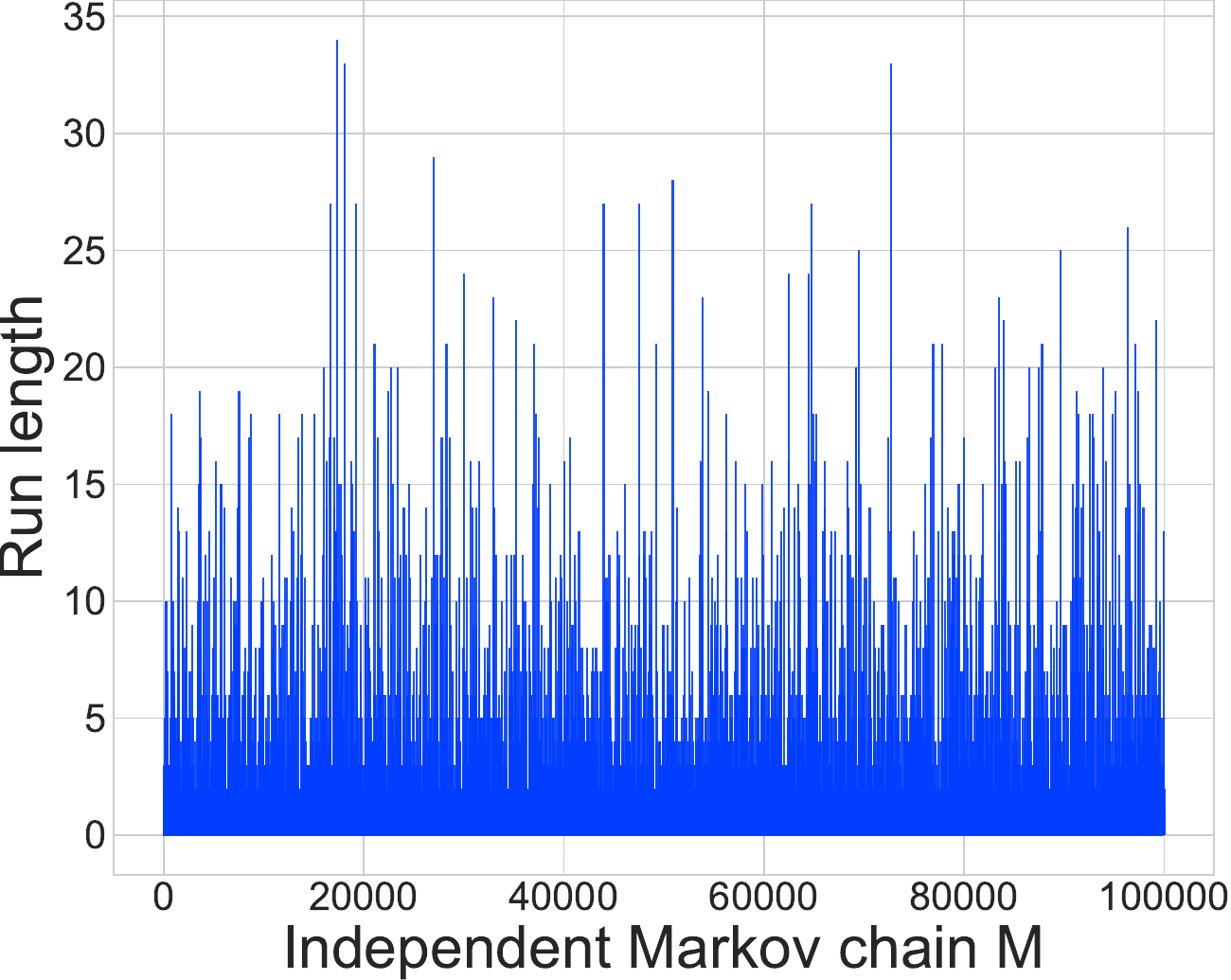}
  \caption{}
  \label{figure:ar_sim_c}
\end{subfigure}
\caption{
Figures \ref{figure:ar_sim_a} and Figure~\ref{figure:ar_sim_b} plot the coordinates of the MSC mean within 2 estimated standard errors compared to standard MC. Figure~\ref{figure:ar_sim_b} plots the length of the Markov sample paths over the independent runs of the Markov chains.
} 
\label{figure:ar_sim}
\end{figure}

\section{Cardiovascular disease prediction with the Pólya-Gamma Gibbs sampler}
\label{section:pg_application}

We look now to a non-trivial example on a real data set with Bayesian logistic regression and the Pólya-Gamma Gibbs sampler \citep{polson:2013}.
A convergence analysis in the single-chain regime has been developed \citep{choi:2013}, but the scaling of the constants in the convergence rate may lead to a long simulation length in moderate to high dimensions.
It may then be an advantage to incorporate parallel computation using many-short-chains for this Markov chain that does not require such a convergence analysis.

Here we use standard notation in Bayesian inference although it overlaps slightly with our previous notation.
Let $(Y_i, X_i)_{i = 1}^n$ with $Y_{i} \in \{0, 1\}$ and $X_i \in \R^d$ and define $Y = (Y_1, \ldots, Y_n)^T \in \{ 0, 1 \}^n$ and $X = (X_1, \ldots, X_n)^T \in \R^{n \times d}$.
For $x \in \R$, define the sigmoid function by $\sigma(x) = (1 + \exp(-x))$.
Consider the Bayesian logistic regression model with a Gaussian prior satisfying
\begin{align*}
&Y_i | X_i, \beta \sim \text{Bern}( \sigma(\beta^T X_i) )
&\beta \sim N(0, \Sigma)
\end{align*}
where $\Sigma \in \R^{d \times d}$ is a symmetric, positive-definite (SPD) covariance matrix.
Define the negative log-likelihood by
$
\l_n(\beta) = \sum_{i = 1}^n \left[ \log(1 + \exp(X_i^T \beta)) - Y_i X_i^T \beta \right].
$
The posterior for this model $\Pi_n$ has a Lebesgue density defined by
\[
\pi_n(\beta)
\propto \exp\left[
-\l_n(\beta)
- \frac{1}{2} \beta^T \Sigma^{-1} \beta
\right].
\]
We will be interested in estimating the posterior mean of the coefficients $\int \beta \Pi_n(d\beta)$ to perform Bayesian inference.

The Pólya-Gamma Gibbs sampler uses Pólya-Gamma random variables to construct an augmented two-variable deterministic scan Gibbs sampler. This results in a Markov chain $(\w_t, \beta_t)_{t = 0}^\infty$ and a marginal chain $(\beta_t)_{t = 0}^\infty$ that has the posterior as its invariant distribution.
Let $\text{PG}(1, b)$ denote the Pólya-Gamma distribution with parameter $b \ge 0$ \citep{polson:2013}.
For $\w \in \R^n$, define 
\begin{align*}
&\Omega = \text{diag}(\w)
&\Sigma(\w) = (X^T \Omega X + \Sigma^{-1} )^{-1}
&&\mu(\w) = \Sigma(\w) X^T (Y - (1/2) 1_d).
\end{align*}
For $t = 1, \ldots$, this Pólya-Gamma Markov chain alternates
\begin{align*}
&\w_t = (\w^1_t, \ldots, \w^n_{t}) | \beta_{t-1}
&\text{with independent} 
&&\w^i_t \sim \text{PG}(1,
| X_i^T \beta_{t-1} |)
\\
&\beta_t | \w_t 
\sim N\left[ \mu(\w_t), \Sigma(\w_t) \right].
\end{align*}

The posterior density is log-concave and we prove a general result to choose the importance sampling proposal in this scenario.
For functions $g : \R^d \to \R$, denote the gradient and Hessian at $x \in \R^d$ by $\nabla g(x)$ and $\nabla^2 g(x)$ respectively.
For an arbitrary matrix $A$, let $\norm{A}_2$ denote the square root of the largest eigenvalue of $A^T A$.
The following Proposition allows choosing a proposal for many log-concave distributions and we will apply it to this example.

\begin{proposition}
\label{proposition:initial}
Let $\l : \R^d \to \R$ and suppose $\l$ is a twice continuously differentiable convex function such that for some $\lambda^* > 0$
\begin{align}
\sup_{x \in \R^d} \norm{ \nabla^2 \l(x) }_2
\le \lambda^*.
\label{eq:Hess_assumption}
\end{align}
Let $\Sigma$ be a SPD matrix and define a function $f$ by  
\begin{align*}
f(x) = \l(x) + \frac{1}{2} x^T \Sigma^{-1} x
\end{align*}
and define a probability density by $\pi(x) = \exp(-f(x)) / Z$ where the normalizing constant $Z = \int \exp(-f(x)) dx$.
Let $x^* = \text{argmin} \{ f(x) : x \in \R^d \}$ and for $h \in (0, 1/2]$, let $Q \sim N(x^*, (1/ 2 + h) \Sigma)$ with density $q$.
Then the upper bound holds
\[
\int_{\R^d} \frac{\pi(x)^2}{q(x)} dx
\le 
\det( \lambda^* \Sigma + I_d)
\left( \frac{1}{ 2 \sqrt{2 h} } + \sqrt{\frac{h}{2}} \right)^{d}.
\]
\end{proposition}

Let $\beta^*_n$ be the maximum of the posterior density $\pi_n$, which always exists.
With Proposition~\ref{proposition:initial} in mind, we will choose the importance sampling proposal $Q_h \equiv N(\beta_n^*, (1/2 + h) \Sigma)$ with $h \in (0, 1/2]$.
Let $\theta = (\theta_1, \ldots, \theta_N)$ where $\theta_j \sim Q_h$ are independent to construct the random initial distribution $\Pi_\theta^N$.
Define $L = \norm{ \Sigma }_2^2 \norm{ X^T (Y - 1/2 1_d) }_2^2$ and for any $r > 1$, define the set 
\begin{align}
C_r = \{ \beta \in \R^d : \norm{ \beta }_2^2 \le r L \}.
\label{eq:C_pg}
\end{align}
A drift condition is shown in Proposition~\ref{proposition:pg_sampler} to this set $C_r$ and yields the main result for the error analysis of the MSC estimator using this Pólya-Gamma sampler.

\begin{proposition}
\label{proposition:pg_sampler} 
With the set $C_r$ defined by \eqref{eq:C_pg}, let $\theta_j \sim Q_h$ be independent for $j = 1, \ldots, N$ and $S^m_{\tau_{C_r}}(\phi) | \theta$ be independent $m = 1, \ldots, M$ using marginal Pólya-Gamma Markov chains $(\beta^m_t)_{t = 1}^{\infty}$ initialized at $\beta^m_0 \sim \Pi_\theta^N$.
Then for all functions with $|\phi(\beta)| \le \sqrt{ \norm{\beta}_2^2 + 1}$,
\begin{align*}
\E\left[ 
\left| 
\frac{1}{M} \sum_{m = 1}^M S^m_{\tau_{C_r}}(\phi) - \int_\X \phi d\Pi 
\right|^2
\right]
\le \left(
\frac{\gamma_r \sqrt{(r + 1)L + 2}}{ \sqrt{M} (2 - \gamma_r) }
+ \frac{[2 L + 1] W_d }{\sqrt{ N } (1 - \gamma_r )}
\right)^2.
\end{align*}
where $\gamma_r = (1 + L)/(1 + rL)$ and
\[
W_d = \det( \norm{X}_2^2 \Sigma/4 + I_d) \left( \frac{1}{ 2 \sqrt{2 h} } + \sqrt{\frac{h}{2}} \right)^{d}.
\]
\end{proposition}

We now investigate the empirical performance on a real data set of the MSC estimator from Proposition~\ref{proposition:pg_sampler}.
The aim is to predict cardiovascular disease from a well-known data set provided by the Cleveland Clinic \citep{detrano:etal:1989}.
The data consists of $303$ patients with binary responses determining if cardiovascular disease is present and $14$ predictor variables based on patient characteristics.
After converting categorical predictors, a total of $21$ covariates are used.
We will investigate 
the coefficients $\beta_1$ indicating male versus female patients, $\beta_2$ for the number of major blood vessels colored by flourosopy, and $\beta_3$ for resting blood pressure in mm/Hg on admission to the hospital.
For the prior, a covariance $\Sigma = 10 I$ is chosen.

We compare the MSC estimation performance to a single-chain Pólya-Gamma Gibbs simulation to estimate the posterior mean of the covariates of interest.
For an additional comparison, a tuned random-walk Metropolis-Hastings (RWM) is also simulated using scaling parameter according to optimal scaling \citep{robe:weak:1997}.
We generate the Pólya-Gamma Gibbs sampler and RWM for $10^6$ iterations starting from the random initial distribution used the MSC estimation.
For the MSC estimate, $N = 10^7$ independent samples are used for the initial distribution and $M = 10^6$ independent Markov chains are used with tuning parameters $h = .49$ and $r = 1.001$.

Figure~\ref{figure:pg_sim} shows the means of the MSC estimates within 2 estimated standard errors of the simulations compared to the means for the single-chain simulations.
Similar performance is observed of all algorithms up to a small Monte Carlo error.
The drift condition here results in a large radius for the set $C_r$ and the simulation results in Markov chain paths of essentially length $1$.
In particular, the estimate in the simulation is an empirical mean of a combination of only $1$ step from the Gibbs sampler. 
However, the initial distribution must be carefully constructed for the MSC estimate when compared to the single-chain MCMC algorithms.

\begin{figure}[t]
\centering
\begin{subfigure}{.32\linewidth}
  \includegraphics[width=\linewidth]{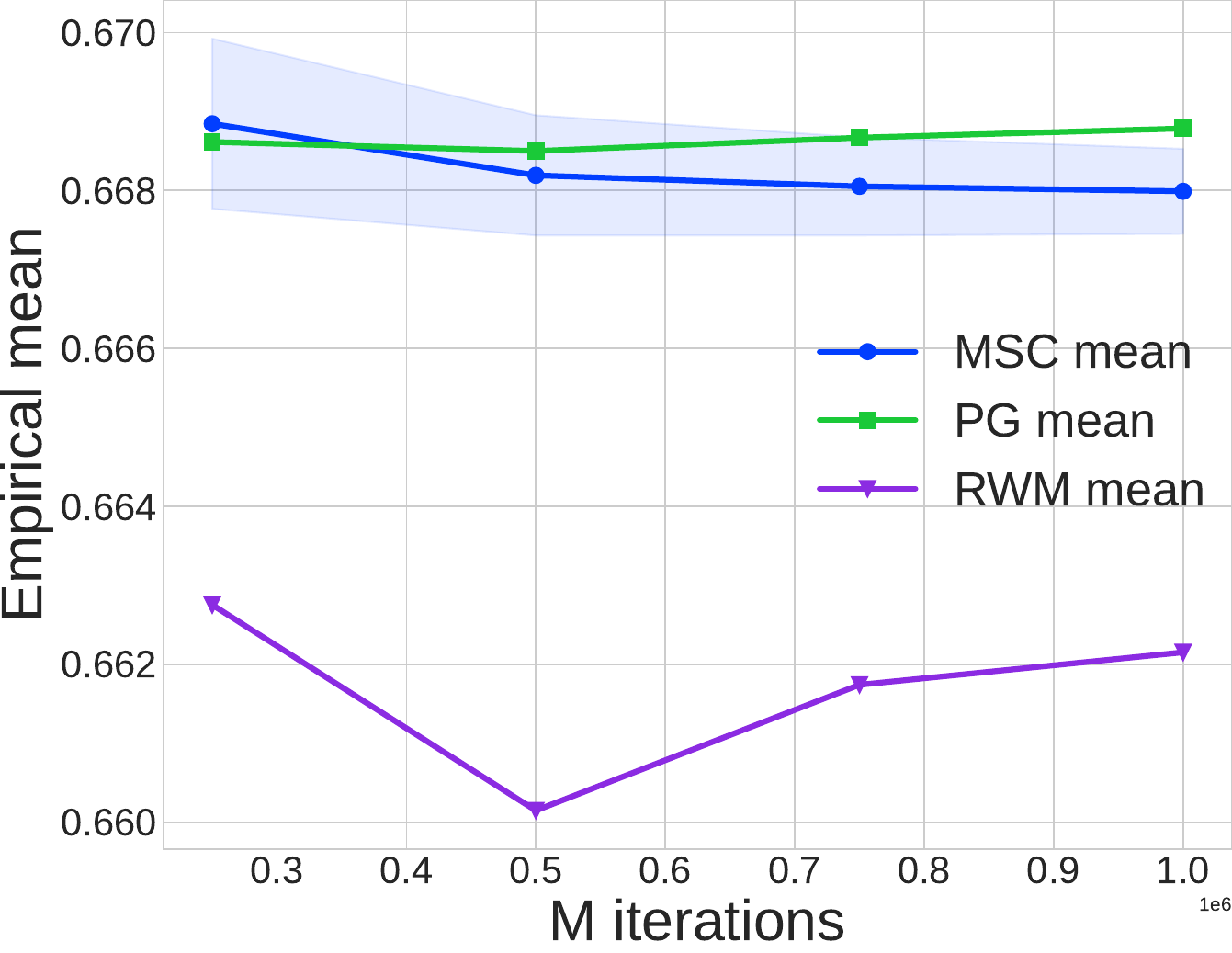}
  \caption{}
  \label{figure:pg_sim_a}
\end{subfigure}
\hspace{.2cm}
\begin{subfigure}{.32\linewidth}
  \includegraphics[width=\linewidth]{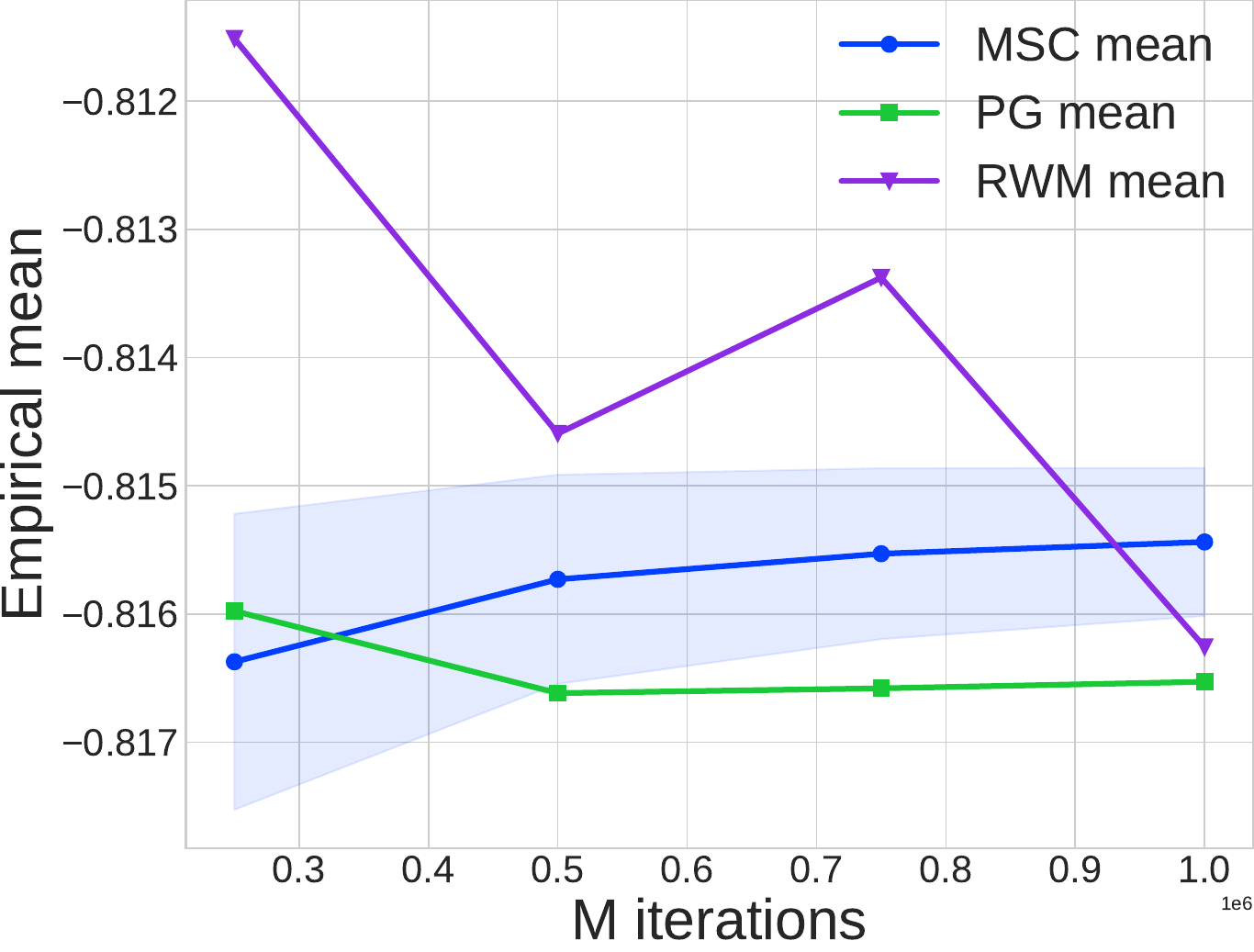}
  \caption{}
  \label{figure:pg_sim_b}
\end{subfigure}
\begin{subfigure}{.32\linewidth}
  \includegraphics[width=\linewidth]{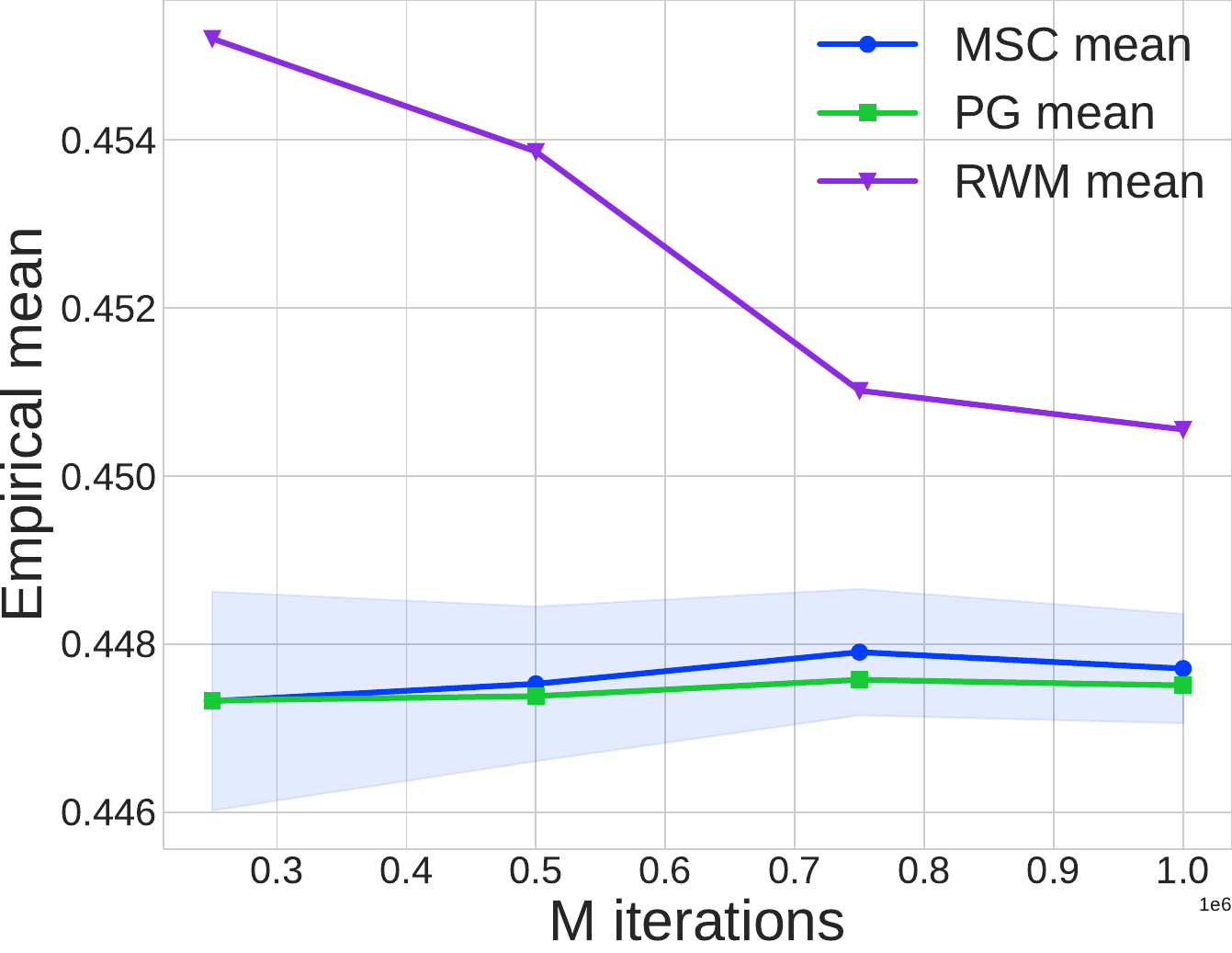}
  \caption{}
  \label{figure:pg_sim_c}
\end{subfigure}
\caption{
The mean within 2 estimated standard deviations of MSC the Pólya-Gamma sampler using compared to the single-chain Pólya-Gamma sampler and RWM algorithm.
} 
\label{figure:pg_sim}
\end{figure}

\section{Discussion and conclusions}
\label{section:conclusion}

The MSC estimator based on the representation \citep[Theorem 10.4.9]{meyn:tweedie:2009} provides interesting insight into the many-short-chains regime.
It is remarkable that the MSC estimator relies only on the drift condition for the Markov chain and not its convergence rate where the latter often scales poorly with the dimension.
Another advantage of the MSC estimator is a simple uncertainty quantification by estimating the standard error since it is based on independent Markov chains.
In comparison, determining the termination time of an MCMC simulation in the single-chain regime is a challenging problem and actively researched.

The parallel MSC estimator developed here may be practically useful when a reasonable importance sampling proposal is available along with an explicit geometric or multiplicative drift condition.
However, this may not always be the case and other parallel MCMC estimation techniques may be suitable instead \citep{pierre:2020}.
For moderately sized problems, the ability to utilize parallel computation is an important property.
However, for some high dimensional problems, the non-asymptotic bounds may not be reasonable unless strong conditions are satisfied so that the non-asymptotic bounds scale appropriately.
There are also some criticisms and open questions on the benefit of using the MSC estimator over self-normalized importance sampling.
The current importance sampling bounds \citep[Theorem 2.1]{agapiou:2017} hold for bounded functions and require stronger conditions for unbounded functions \citep[Theorem 2.3]{agapiou:2017}.
An alternative MSC estimator can also be defined based on a second representation of the invariant measure and this appears to bear some further relations to importance sampling \citep[Theorem 10.4.9]{meyn:tweedie:2009}.

There appear to be many future research directions.
For example, the optimal set $C$ and drift condition is not understood, but this is also the case in convergence analysis using drift and minorization conditions for Markov chains.
Empirically, the MSC estimator has different run lengths based on the size of the set $C$ and benefits and drawbacks of this still still remain somewhat elusive. 
Another possible research direction is to develop error bounds for not only estimation but not for stronger discrepancies such as Wasserstein distances and other probability discrepancies.
However, the lower bounds for such distances are poor in high dimensions \citep{dudley:1969}.

\section*{Acknowledgements}
Thanks to Radu Craiu, Jeffrey Rosenthal, Qiang Sun, Galin Jones, and many others at the University of Toronto for insightful conversations while writing this manuscript.

\section*{Supplementary material}

The Pypolyagamma library \url{https://github.com/slinderman/pypolyagamma} is used by Scott Linderman to generate Pólya random variables.
Simulation and all code is available
\url{https://github.com/austindavidbrown/msc-estimator}.
The cardiovascular data is available at the UCI database \url{https://archive.ics.uci.edu/dataset/45/heart+disease}.

\appendix

\section{Supporting technical results for Section~\ref{section:mse_error_bound}}

The following bound is well-known.
\begin{lemma}
\label{lemma:f_integrable}
Assume the drift condition \eqref{assumption:drift} holds.
Then
\[
\int_{\X} f d\Pi
\le \frac{
K
}{
1 - \gamma
}.
\]
\end{lemma}
\begin{proof}
The drift condition implies
\begin{align*}
P^n V(x) - V(x)
&= \sum_{k = 1}^{n - 1} \left[ 
P^{k + 1}V(x) - P^{k}V(x)
\right]
\le -(1 - \gamma) \sum_{k = 0}^{n - 1} P^k f(x) + n K.
\end{align*}
By Fatou's lemma \citep[Theorem 2.8.3]{bogachev:2007}
\begin{align*}
\int_{\X} f d\Pi
= \limsup_n \int_{\X} \frac{1}{n} \sum_{k = 0}^{n - 1} P^k f(x) d\Pi
\le \int_{\X} \limsup_n \frac{1}{n} \sum_{k = 0}^{n - 1} P^k f(x) d\Pi
\le \frac{
K
}{
1 - \gamma
}.
\end{align*}
\end{proof}

The following bound in Lemma~\ref{lemma:drift_martingale_ub} is implied by the drift condition.
The result is similar to \citep[Proposition 11.3.2]{meyn:tweedie:2009}.
Denote $a \wedge b, a \vee b$ for the minimum and maximum of two numbers $a, b$.

\begin{lemma}
\label{lemma:drift_martingale_ub}
If the drift condition \eqref{assumption:drift} holds, then for every $x \in C$,
\[
\sup_{|\phi| \le f} (\K_C |\phi|)(x)
\le \frac{
V(x) - 1 - ( 1 - \gamma ) f(x) + 2 K
}{1 - \gamma_R}
\]
where $\gamma_R = \gamma + K/R$.
\end{lemma}

\begin{proof}[Proof of Lemma~\ref{lemma:drift_martingale_ub}]
The drift condition \eqref{assumption:drift} implies
\[
\P V(x) - V(x) 
\le -( 1 - \gamma_R ) f(x) + K (1 - \frac{f(x)}{R})
\le -( 1 - \gamma_R ) f(x) + K I_{C}(x).
\]
For $n \in \Z_+$, define
\begin{align}
M_{n + 1}
= V(X_{n + 1})
+ ( 1 - \gamma_R ) \sum_{k = 1}^{n} f(X_k) - K \sum_{k = 1}^{n} I_C(X_k).
\label{eq:martingale}
\end{align}
Using the drift condition, $M_{n + 1}$ is a supermartingale with respect to the standard filtration $\F_{n + 1}$ for $X_0, X_1, \ldots, X_{n + 1}$.
By the optional sampling theorem \citep[Theorem 9.12]{Kallenberg2021}, $M_{n \wedge \tau_C + 1}$ is also a supermartingale with its respective filtration and
\begin{align*}
&( 1 - \gamma_R ) \E\left[ \sum_{k = 1}^{\tau_C \wedge n} f(X_n) 
\bigm| X_0 = x \right]
+ 1 - K
\le \E[M_{\tau_c \wedge n + 1} \bigm| X_0 = x]
\\
&\le \E[M_2 \bigm| X_0 = x]
\le \E[V(X_1) \bigm| X_0 = x].
\end{align*}
Since $f \ge 1$, using monotone convergence \citep[Theorem 2.8.2]{bogachev:2007},
\begin{align*}
\E\left[ \sum_{k = 1}^{\tau_C} f(X_n) 
\bigm| X_0 = x \right]
\le \frac{
\E[V(X_1) \bigm| X_0 = x] - 1 + K
}{1 - \gamma_R}
\\
\le \frac{
V(x) - 1 - (1 - \gamma) f(x) + 2 K
}{1 - \gamma_R}.
\end{align*}
\end{proof}

The following result shows that $\Pi$ is invariant for $\K_C$ and we have an identity for many integrals with respect to the invariant measure under the drift condition \eqref{assumption:drift}.

\begin{theorem}
\label{theorem:representation}
Assume the Markov kernel $\P$ has unique invariant measure $\Pi$ and the drift condition \eqref{assumption:drift} holds. 
If $\int_C V d\Pi < \infty$, then for every $\phi : \X \to \R$ such that $|\phi| \le f$, we have $\int_\X \phi d\Pi < \infty$ and
\begin{align*}
\int_\X \K_C \phi d\Pi
= \int_\X \phi d\Pi.
\end{align*}
\end{theorem}

\begin{proof}[Proof for Theorem~\ref{theorem:representation}] 
Under the drift condition, then irreducibile Markov chains with unique invariant distribution have a representation of the invariant measure $\Pi$ for non-negative simple functions \citep[Theorem 10.4.9]{meyn:tweedie:2009}. 
The assumption of an irreducible Markov chain is not required as seen in the proof \citep[Theorem 10.4.9]{meyn:tweedie:2009} and we include the proof only for completeness. 
Let $\phi : \X \to [0, \infty]$ be simple so that $\phi = \sum_{k = 1}^J a_k I_{B_k}$ for some constants $a_k \ge 0$ and sets $B_k$.
By invariance and induction it is readily shown that 
\[
\int \phi d\Pi
\ge \sum_{k = 1}^{n} \int_C \E\left[ \phi(X_k) I_{\tau_C \ge k} \bigm| X_0 = x \right] \Pi(dx).
\]
For sets $B \subseteq \X$, this implies
$
\Pi(B \cap C)
\ge \int_C \Prob\left[ X_{\tau_C} \in B, \tau_C < \infty \bigm| X_0 = x \right] \Pi(dx).
$
Suppose by contradiction that there is some set $B_0 \subseteq C$ with 
\[
\Pi(B_0) > \int_C \Prob\left[ X_{\tau_C} \in B_0, \tau_C < \infty \bigm| X_0 = x \right] \Pi(dx).
\]
We must have
\begin{align*}
\Pi(C)
= \Pi(B_0) + \Pi(C \cap B_0^c)
> \int_C \Prob\left[ \tau_C < \infty \bigm| X_0 = x \right]
\Pi(dx).
\end{align*}
Since $f \ge 1$, the drift condition implies $\Prob\left[ \tau_C < \infty \bigm| X_0 = x \right] = 1$ for all $x \in C$ and so this is a contradiction and for every set $B \subseteq \X$
\[
\Pi(B \cap C)
\le \int_C \Prob\left[ X_{\tau_C} \in B, \tau_C < \infty \bigm| X_0 = x \right] \Pi(dx)
\le \Pi(B \cap C).
\]
Combining these inequalities implies
\begin{align}
\int_C \E\left[ \phi(X_{\tau_C})I_{\tau_C < \infty} \bigm| X_0 = x \right] \Pi(dx)
= \int_C \phi d\Pi.
\label{eq:hit_identity}
\end{align}
Combining \eqref{eq:hit_identity} and the Markov property
\begin{align*}
&\int_C \E\left[ \sum_{k = 1}^{\tau_C} \P \phi(X_k) \bigm| X_0 = x \right]
\Pi(dx)
\\
&= \sum_{k = 1}^{\infty} \int_C \E\left[ \P \phi(X_{k + 1}) I_{\tau_C \ge k + 1} \bigm| X_0 = x \right]
\Pi(dx)
+ \sum_{k = 1}^{\infty} \int_C \E\left[ \P \phi(X_{k}) I_{\tau_C = k} \bigm| X_0 = x  \right]
\Pi(dx)
\\
&= \int_C \E\left[ \sum_{k = 2}^{\tau_C} \phi(X_{k}) \bigm| X_0 = x  \right]
\Pi(dx)
+ \int_C \E\left[ \P\phi(X_{\tau_C}) I_{\tau_C < \infty} \bigm| X_0 = x  \right]
\Pi(dx)
\\
&= \int_C \E\left[ \sum_{k = 2}^{\tau_C} \phi(X_{k}) \bigm| X_0 = x \right]
\Pi(dx)
+ \int_C \P \phi d\Pi
\\
&= \int_C \E\left[ \sum_{k = 1}^{\tau_C} \phi(X_{k}) \bigm| X_0 = x \right]
\Pi(dx).
\end{align*}

Since $\Pi$ is the unique invariant measure and under the drift condition, we have
\[
\int_\X \phi d\Pi
= \int_C \E\left[ \sum_{k = 1}^{\tau_C} \phi(X_k) \bigm| X_0 = x \right]
\Pi(dx).
\]
This representation extends to simple functions $\phi : \X \to \R$ taking the positive and negative parts.

It remains to extend this representation to more general functions.
Now let $\phi : \X \to \R$ such that $|\phi| \le f$ and Lemma~\ref{lemma:f_integrable} implies
$
\int_\X | \phi | d\Pi
\le K/(1 - \gamma).
$
Since $\phi$ is measurable, we can choose a sequence of simple functions $\phi_n \to \phi$ pointwise with $|\phi_n| \le |\phi_{n + 1}|$.
Using Lemma~\ref{lemma:drift_martingale_ub} and dominated convergence \citep[Theorem 2.8.1]{bogachev:2007}, for each $x \in C$
$
\lim_{n \to \infty} (\K_C \phi_n)(x) 
= (\K_C \phi)(x).
$
Using Lemma~\ref{lemma:drift_martingale_ub} and since we have assumed $\int_C V d\Pi < \infty$, we have by dominated convergence \citep[Theorem 2.8.1]{bogachev:2007},
\begin{align*}
\int_\X \K_C \phi d\Pi
&= \int_C \lim_{n \to \infty} (\K_C \phi_n) d\Pi
= \lim_{n \to \infty} \int_C (\K_C \phi_n) \Pi(dx)
\\
&= \lim_{n \to \infty} \int_\X \phi_n d\Pi
= \int_\X  \lim_{n \to \infty} \phi_n d\Pi
= \int_\X \phi d\Pi.
\end{align*}

\end{proof}

\begin{proof}[Proof of Proposition~\ref{proposition:bias}]
Lemma~\ref{lemma:drift_martingale_ub} implies
\begin{align}
\sup_{x \in \X}| (\K_C \phi)(x) |
\le \frac{ A_R }
{ 1 - \gamma_R }.
\label{eq:lb_ub_f}
\end{align}
Since we assumed $\sup_{x \in C} V(x) < \infty$, then $A_R < \infty$ and we have shown that $\K_C \phi$ is bounded.
Here we used that $x$ is restricted to $C$ due to the definition of $\K_C$.
Since we have assumed the drift condition, the invariant distribution is unique, and $\sup_{x \in C} V(x) < \infty$, then Theorem~\ref{theorem:representation} implies the identity
\[
\E[ (\K_C \phi)(Y_1) w(Y_1)]
= \int_C \K_C \phi d\Pi
= \int_\X \phi d\Pi.
\]

By \citep[Theorem 2.1]{agapiou:2017}, we have the error bound for self-normalized importance sampling with the bounded function $\K_C \phi$.
We can slightly improve the bound, but the proof technique is essentially the same.
Using Cauchy-Schwarz, for any $\alpha > 0$, we have the upper bound
\begin{align}
&\E\left| 
\sum_{i = 1}^N w^N_Y(Y_i) (\K_C \phi)(Y_i)
- \E[ (\K_C \phi) w]
\right|^2
\\
&\le 
(1 + \alpha) \E\left| 
\sum_{i = 1}^N w^N_Y(Y_i) (\K_C \phi)(Y_i)
- \sum_{i = 1}^N \frac{w(Y_i)}{N} (\K_C \phi) (Y_i)
\right|^2
\\
&\hspace{1em}+ (1 + \alpha^{-1}) \E\left| 
\frac{1}{N} \sum_{i = 1}^N w(Y_i) (\K_C \phi)(Y_i)
- \E[(\K_C \phi) w]
\right|^2.
\label{eq:is_bound}
\end{align}
This second term in \eqref{eq:is_bound} is the variance of independent random variables:
\begin{align*}
\Var \left[ \frac{1}{N} \sum_{i = 1}^N (\K_C \phi) (Y_i) w(Y_i) \right]
&= \frac{1}{N} \Var\left[ (\K_C \phi)(Y) w(Y) \right]
\le [ \sup_{x \in \X} | (\K_C \phi)(x) | ]^2  \frac{\E\left[ w(Y_1)^2 \right]}{N}.
\end{align*}
For the first term in \eqref{eq:is_bound} , we have
\begin{align*}
&\left| 
\sum_{i = 1}^N
\frac{w(Y_i)}{\sum_j w(Y_j)} \K_C \phi(Y_i)
- \sum_{i = 1}^N \frac{w(Y_i)}{N} \K_C \phi(Y_i)
\right|
\\
&= \left| 
\sum_{i = 1}^N \K_C \phi(Y_i) \frac{w(Y_i)}{\sum_j w(Y_j)}
\left[ 
\frac{1}{N} \sum_{i = 1}^N w(Y_i) - 1
\right]
\right|
\\
&\le \max_{i \in 1, \ldots, N} | (\K_C \phi) (X_i) |
\left|
\frac{1}{N} \sum_{i = 1}^N w(Y_i) - 1
\right|.
\end{align*}
By a standard Monte Carlo argument,
$
\Var[ \overline{w} ]
= \Var[w(Y_1)] / N,
$
it follows then that
\begin{align*}
&\E\left| 
\sum_{i = 1}^N
\frac{w(Y_i)}{\sum_j w(Y_j)} (\K_C \phi)(Y_i)
- \sum_{i = 1}^N \frac{w(Y_i)}{N} (\K_C \phi)(Y_i)
\right|^2
\le [ \sup_{x \in \X} | (\K_C f)(x) | ]^2 \frac{\Var[w]}{N}.
\end{align*}
Combining these bounds choosing $\alpha = 1$, we have
\begin{align*}
&\E\left| 
\sum_{i = 1}^N w^N_Y(Y_i) (\K_C \phi)(Y_i)
- \E[ (\K_C \phi) w]
\right|^2
\\
&\le 
\frac{
2 \int w d\Pi
}{
N
}
\frac{A_R^2}{(1 - \gamma_R)^2}
+ \frac{
2 [\int w d\Pi - 1 ]
}{
N
}
\frac{A_R^2}{(1 - \gamma_R)^2}
\\
&\le 
\frac{4 A_R^2 \int w d\Pi}{N (1 - \gamma_R)^2}
- \frac{2 A_R^2}{N (1 - \gamma_R)^2}.
\end{align*}

\end{proof}

\begin{proof}[Proof of Proposition~\ref{proposition:variance}] 
Since $S_{\tau_C}^m(\phi) | Y$ are assumed independent, we have $\E[S_{\tau_C}^m(\phi) | Y ] = \int \K_C \phi d\Pi^N_Y$ and the variance bound
\[
\E\left[ \left| \frac{1}{M} \sum_{m = 1}^M S^m_{\tau_C}(\phi) - \int \K_C \phi d\Pi^N_Y \right|^2 \Bigm| Y \right]
\le \frac{
\E\left[  | S^m_{\tau_C}(\phi) |^2 \Bigm| Y \right]
}{M}.
\]
Let $\F_n$ denote the standard filtration with respect to $X_0, \ldots, X_n$.
Repeated application of the geometric drift condition for $x \in \X$ gives
\begin{align*} 
\E[ V(X_k) I_{\tau_C \ge k} \bigm| X_0 = x, Y]
&= \E\left[ I_{\tau_C \ge k} \E( V(X_k) \bigm| \F_{k - 1}, X_0 = x, Y ) | Y \right]
\\
&\le \E\left[ \gamma_R V(X_{k - 1})  I_{\tau_C \ge k} + KI_C(X_{k - 1}) I_{\tau_C \ge k} \bigm| X_0 = x, Y \right]
\\
&\le \gamma_R \E\left[ V(X_{k - 1})  I_{\tau_C \ge k-1} \bigm| X_0 = x, Y \right]
\\
&\le \gamma_R^k [V(x) + KI_C(x)].
\end{align*}
Here we used that for any $k \in \Z_+$, $I_{\tau_C \ge k} \le I_{\tau_C \ge k - 1}$ and $\{ \tau_C \ge k \}$ is $\F_{k - 1}$ measurable.

Using the Cauchy-Schwarz inequality,
\begin{align*}
&\E\left[ \left( \sum_{k = 1}^{n} \phi(X_k) I_{\tau_C \ge k} \right)^2 \bigm| X_0, Y \right]
\\
&= \sum_{k = 1}^n \sum_{\l = 1}^{n} \E\left[ \phi(X_k) I_{\tau_C \ge k} 
\phi(X_\l) I_{\tau_C \ge \l}
\bigm| X_0, Y 
\right]
\\
&\le \sum_{k = 1}^n \sum_{\l = 1}^{n} \sqrt{ \E\left[ \phi(X_k)^2 I_{\tau_C \ge k} \bigm| X_0, Y  \right] }
\sqrt{ 
\E\left[ 
\phi(X_\l)^2 I_{\tau_C \ge \l}
\bigm| X_0, Y
\right]
}.
\end{align*}
Now applying the drift condition,
\begin{align*}
\E\left[ \left( \sum_{k = 1}^{\tau_C \wedge n} \phi(X_k) \right)^2 \bigm| X_0, Y \right]
&\le \sum_{k = 1}^n \sum_{\l = 1}^{n} \sqrt{ \E\left[ V(X_k) I_{\tau_C \ge k} 
\bigm| X_0, Y \right] }
\sqrt{ 
\E\left[ 
V(X_\l) I_{\tau_C \ge \l}
\bigm| X_0, Y
\right]
}
\\
&\le [V(X_0) + K] \sum_{k = 1}^n \sum_{\l = 1}^{n} 
\gamma_R^{k/2} 
\gamma_R^{\l/2}
\\
&\le [V(X_0) + K] \sum_{k = 1}^n \sum_{\l = 1}^{n} 
\gamma_R^{k/2} 
\gamma_R^{\l/2}
\\
&\le [V(X_0) + K] \left[ \frac{\gamma_R/2}{1 - \gamma_R/2} \right]^2.
\end{align*}
Using Fatou's lemma \citep[Theorem 2.8.3]{bogachev:2007} and since $\sup_{x \in C} V(x) \le R$, we conclude
\[
\E\left[ S^m_{\tau_C}(\phi)^2 \bigm| X_0, Y  \right]
\le [R + K] \left[ \frac{\gamma_R/2}{1 - \gamma_R/2} \right]^2.
\]
Taking the iterated expectation with respect to $X_0 \sim \Pi^N_Y$ and then $Y_j \sim Q$ for $j = 1, \ldots, N$, completes the proof.
\end{proof}

\begin{proof}[Proof of Theorem~\ref{theorem:mse_bound}]
Since $V = f \ge 1$, then $|\phi| \le \sqrt{V} \le V$ and also $\sup_{x \in C} V(x) \le R$, so we are able to apply both upper bounds in Proposition~\ref{proposition:bias} and Proposition~\ref{proposition:variance}.
By Lemma~\ref{lemma:f_integrable},
$
\int_\X | \phi | d\Pi
\le \int_\X V d\Pi
\le K/(1 - \gamma).
$
Using Cauchy-Schwarz for any $\alpha > 0$, we have the upper bound
\begin{align*}
\E\left[ \left| \frac{1}{M} \sum_{m = 1}^{M} S_{\tau_C}^m(\phi)
- \int_\X \phi d\Pi \right|^2 \right]
&\le (1 + \alpha) \E\left[ \left| \frac{1}{M} \sum_{m = 1}^{M} S_{\tau_C}^m(\phi)
- \int_\X \K_C \phi d\Pi^N_Y \right|^2 \right]
\\
&\hspace{1em}+ (1 + \alpha^{-1}) \E\left[ \left| \int_\X \K_C \phi d\Pi^N_Y
- \int_\X \phi d\Pi \right|^2 \right].
\end{align*}
Using Proposition~\ref{proposition:bias}, Proposition~\ref{proposition:variance} and optimizing over $\alpha$,
\begin{align*}
\E\left[ \left| \frac{1}{M} \sum_{m = 1}^{M} S_{\tau_C}^m(\phi)
- \int_\X \phi d\Pi \right|^2 \right]
&\le (1 + \alpha) \frac{[R + K]}{M} \left[ \frac{ \gamma_R/2 }{ 1 - \gamma_R/2 } \right]^2
+ (1 + \alpha^{-1}) \frac{4 A_R^2 \int w d\Pi}{N(1 - \gamma_R )^2}
\\
&\le 
\left(
\frac{\gamma_R \sqrt{R + K}}{ \sqrt{M} (2 - \gamma_R) }
+ \frac{ 2 A_R \sqrt{\int w d\Pi} }{\sqrt{ N } (1 - \gamma_R )}
\right)^2.
\end{align*}
\end{proof}

\section{Supporting technical results for Section~\ref{section:concentration_mult}}

Similar to Lemma~\ref{lemma:drift_martingale_ub}, the following upper bound is implied by the drift condition.

\begin{lemma}
\label{lemma:mult_drift_martingale_ub}
If the multiplicative drift condition \eqref{assumption:drift_mult} holds, then for every $x \in C$
\begin{align*}
\E \left( \exp\left[ (1 - \gamma_R) \sum_{k = 1}^{\tau_C} f(X_k) \bigm| X_0 = x \right]
\right)
&\le \exp\left[ 
V(x) - (1 - \gamma) f(x) + 2 K
\right]
\end{align*}
where $\gamma_R = \gamma + K/R$.
\end{lemma}
\begin{proof}
For any $x \in \X$, the drift condition implies
\begin{align*}
\log[ \exp(-V(x)) (\P \exp(V)) (x) ]
&\le - (1 - \gamma - K/R) f(x) + K(1 - f(x)/R)
\\
&\le - (1 - \gamma_R) f(x) + K I_C(x).
\end{align*}
Define
\[
M_{n + 1} = \exp\left[ V(X_{n + 1}) + (1 - \gamma_R) \sum_{k = 1}^{n} f(X_k) - K \sum_{k = 1}^{n} I_C(X_k) \right].
\]
For every $n \in \Z_+$, let $\F_n$ be the standard filtration for $X_0, X_1, \ldots, X_n$ and then using the drift condition,
$
\E_x(M_{n + 1}| \F_n)
\le M_n
$
and $M_n$ is a supermartingale.
By the optional sampling theorem \citep[Theorem 9.12]{Kallenberg2021},
\begin{align*}
&\E \left( \exp\left[ - K + (1 - \gamma_R) \sum_{k = 1}^{\tau_C \wedge n} f(X_k) \bigm| X_0 = x \right]
\right)
\le \E(M_{\tau_C \wedge n} \bigm| X_0 = x )
\\
&\le \E(M_2 \bigm| X_0 = x )
\\
&\le \E(\exp(V(X_1)) \bigm| X_0 = x ).
\end{align*}
Taking the limit and using Fatou's lemma \citep[Theorem 2.8.3]{bogachev:2007},
\begin{align*}
\E \left( \exp\left[(1 - \gamma_R) \sum_{k = 1}^{\tau_C} f(X_k) \bigm| X_0 = x \right]
\right)
&\le \E\left[ \exp( V(X_1)) + K ) \bigm| X_0 = x \right]
\\
&\le \exp\left[ 
V(x) - (1 - \gamma) f(x) + 2 K
\right].
\end{align*}
\end{proof}

We have the identity under the multiplicative drift.
The following proof is similar to Theorem~\ref{theorem:representation}.

\begin{theorem}
\label{theorem:multi_representation}
Assume the Markov kernel $\P$ has unique invariant measure $\Pi$ and the multiplicative drift condition \eqref{assumption:drift_mult} holds. 
If $sup_{x \in C} V(x) < \infty$, then for every $\phi : \X \to \R$ such that $|\phi| \le f$, then $\int_\X \phi d\Pi < \infty$ and
\begin{align*}
\int_\X \K_C \phi d\Pi
= \int_\X \phi d\Pi.
\end{align*}
\end{theorem}
\begin{proof}[Proof of Theorem~\ref{theorem:multi_representation}]
Using the multiplicative drift and since $\Pi$ is the unique invariant measure and under the drift condition, we can show as in Theorem~\ref{theorem:representation} for every simple function $\phi : \X \to \R$,
\[
\int_\X \phi d\Pi
= \int_C \K_C \phi d\Pi.
\]
Now let $\phi : \X \to \R$ such that $|\phi| \le f$.
Since $\phi$ is measurable, we can choose a sequence of simple functions $\phi_n \to \phi$ pointwise with $|\phi_n| \le |\phi_{n + 1}|$.
Using Lemma~\ref{lemma:mult_drift_martingale_ub} and $1 + x \le e^x$ for $x \ge 0$,
\begin{align*} 
(1 - \gamma_R) (\K_C f)(x) 
&\le \E\left\{ \exp\left[ (1 - \gamma_R) \sum_{k = 1}^{\tau_C} f(X_k) \right] \bigm| X_0 = x \right\}
\\
&\le \exp(V(x) + 2K).
\end{align*}
Since we have assumed $\sup_{x \in C} V(x) < \infty$, then $\sup_{x \in C} (\K_C f)(x) < \infty$.
By dominated convergence \citep[Theorem 2.8.1]{bogachev:2007},
\begin{align*}
\int_\X \K_C \phi d\Pi
= \int_\X \lim_{n \to \infty} \K_C \phi_n(x) \Pi(dx)
=  \lim_{n \to \infty} \int_\X \K_C \phi_n(x) \Pi(dx)
= \int_\X \phi d\Pi.
\end{align*}

\end{proof}

\begin{proof}[Proof of Proposition~\ref{proposition:bias_bounded}]
We have the upper bound
\begin{align*}
&\left| 
\sum_{i = 1}^N
\frac{w(Y_i)}{\sum_j w(Y_j)} (\K_C \phi)(Y_i)
- \sum_{i = 1}^N \frac{w(Y_i)}{N} (\K_C \phi)(Y_i)
\right|
&\le \max_{i \in 1, \ldots, N} (\K_C f)(X_i)
\left|
\frac{1}{N} \sum_{i = 1}^N w(Y_i) - 1
\right|.
\end{align*}
This implies the upper bound
\begin{align*}
&\left| 
\sum_{i = 1}^N w^N_Y(Y_i) (\K_C \phi)(Y_i)
- \E[(\K_C\phi)(Y_1) w(Y_1)]
\right|
\\
&\le 
\left| 
\sum_{i = 1}^N w^N_Y(Y_i) (\K_C \phi)(Y_i)
- \sum_{i = 1}^N \frac{w(Y_i)}{N} (\K_C \phi)(Y_i)
\right|
\\
&\hspace{1em}+ \left| 
\frac{1}{N} \sum_{i = 1}^N w(Y_i) (\K_C \phi)(Y_i)
- \E[(\K_C \phi)(Y_1) w(Y_1)]
\right|
\\
&\le 
\sup_{x} (\K_C f)(x) 
\left| 
\frac{1}{N} \sum_{i = 1}^N w(Y_i)
- 1
\right|
\\
&\hspace{1em}+ \left| 
\frac{1}{N} \sum_{i = 1}^N w(Y_i) (\K_C \phi)(Y_i)
- \E[(\K_C \phi)(Y_1) w(Y_1)]
\right|.
\end{align*}
Since for $a, b \ge 0$, $a,b < \e/2$ implies $a + b < \e$, a union probability bound and Hoeffding's inequality \citep{hoeffding:1963},
\begin{align*}
&\Prob\left( 
\left| 
\sum_{i = 1}^N w^N_Y(Y_i) (\K_C \phi)(Y_i)
- \E[(\K_C \phi)(Y_1) w(Y_1)]
\right|
\ge \e
\right)
\\
&\le \Prob\left( 
\sup_{x} (\K_C f)(x) 
\left| 
\frac{1}{N} \sum_{i = 1}^N w(Y_i)
- 1
\right|
\ge \e/2
\right)
\\
&\hspace{1em}+ \Prob\left( 
\left| 
\frac{1}{N} \sum_{i = 1}^N w(Y_i) (\K_C \phi)(Y_i)
- \E[(\K_C \phi)(Y_1) w(Y_1)]
\right|
\ge \e/2
\right)
\\
&\le 4 \exp\left( 
-\frac{\e^2 N}{2 \left[ \sup_x (\K_C f)(x) \right]^2 w_*^2 }
\right).
\end{align*}
Since the multiplicative drift condition holds, $\Pi$ is the unique invariant distribution and $\sup_{x \in C} V(x) < \infty$, then Theorem~\ref{theorem:multi_representation} implies
\[
\E[(\K_C \phi)(Y_1) w(Y_1)]
= \int_C (\K_C \phi) d\Pi
= \int_\X \phi d\Pi.
\]
By Lemma~\ref{lemma:mult_drift_martingale_ub} and using $1 + x \le e^x$, then for $x \in C$,
\begin{align*}
(1 - \gamma_R) (\K_C |\phi|)(x)
&\le (1 - \gamma_R) (\K_C f)(x)
\\
&\le  
\exp\left[ V(x) - (1 - \gamma)f(x) + 2K \right].
\end{align*}
\end{proof}

\begin{proof}[Proof of Proposition~\ref{proposition:chernoff}]
Let $\lambda > 0$ and let $\phi : \X \to \R$ with $|\phi| \le (1 - \gamma_R) f /(1 + \lambda)$.
Since $x^2/2 + 1 \le e^x$ for $x \ge 0$ and using a second order Taylor expansion of $\exp(\lambda \cdot)$,
\begin{align*}
&\E\left[\exp\left( \lambda S^m_{\tau_C}(\phi) - \lambda \int \K_C \phi d\Pi^N_Y \right) \bigm| Y \right]
\\
&= 1 + \lambda^2 \E\left[ \int_0^1 \exp\left[ t \lambda \left( S^m_{\tau_C}(\phi) - \int \K_C \phi d\Pi^N_Y \right) 
\right] \left( S^m_{\tau_C}(\phi) - \int \K_C \phi d\Pi^N_Y \right)^2 (1 - t)dt \bigm| Y \right]
\\
&\le 1 + \lambda^2 
\E\left[ 
\exp\left(
(1 + \lambda) S^m_{\tau_C}(|\phi|) + (1 + \lambda) \int \K_C |\phi| d\Pi^N_Y  
\right)
\bigm| Y
\right].
\\
&\le 1 + \lambda^2 
\E\left[ 
\exp\left(
(1 - \gamma_R) S^m_{\tau_C}(f) + (1 - \gamma_R) \int \K_C f d\Pi^N_Y  
\right)
\bigm| Y
\right].
\end{align*}
By Lemma~\ref{lemma:mult_drift_martingale_ub}, we have the upper bound
\[
\sup_{x \in C} \E\left[ 
(1 - \gamma_R) \exp(S^m_{\tau_C}(f)) 
\Bigm| X_0 = x, Y \right]
\le  
\exp[ B_R ].
\]
Since $\exp$ is convex, we have by Jensen's inequality,
\[
\exp\left[ (1 - \gamma_R) \int \K_C f d\Pi^N_Y \right]
\le \int \E\left\{ \exp\left[ (1 - \gamma_R) S_{\tau_C}(f) \bigm| X_0 = x, Y \right]
\right\} \Pi^N_Y(dx)
\le \exp[ B_R ].
\]
Using this upper bound and $x + 1 \le e^x$ for $x \ge 0$,
\begin{align*}
\E\left[\exp\left( \lambda S_{\tau_C}(\phi) - \lambda \int \K_C \phi d\Pi^N_Y \right) \bigm| Y \right]
&\le 1 + \lambda^2 
\exp( 2 B_R )
&\le \exp\left[ \lambda^2 \exp(2 B_R) \right].
\end{align*}
We have shown this is a sub-Gaussian random variable.
Since $S_{\tau_C}^m(\phi) | Y$ are assumed independent and $\E[S_{\tau_C}^m(\phi) | Y ] = \int \K_C \phi d\Pi^N_Y$, taking the Chernoff bound with the optimal $\lambda = \e \exp(-2 B_R) / 2$,
\begin{align*}
\Prob\left[ \frac{1}{M} \sum_{m = 1}^M S^m_{\tau_C}(\phi) - \int \K_C \phi d\Pi^N_Y \ge \e \Bigm| Y \right]
&\le \exp\left[ M \e 
\left\{ 
\lambda^2 
\frac{\exp(2 B_R)}{\e} - \lambda
\right\}
\right]
\\
&\le \exp\left[ -\frac{M \e^2}{ 4 \exp(2 B_R)}
\right].
\end{align*}
Taking the union probability bound and substituting for the function $\phi = (1 - \gamma_R) \psi / (1 + \lambda)$ with $|\psi| \le f$,
\begin{align*}
\Prob\left[ \left| \frac{1}{M} \sum_{m = 1}^M S^m_{\tau_C}(\psi) - \int \K_C \psi d\Pi^N_Y \right| \ge \e \Bigm| Y \right]
\le 2 \exp\left[ -\frac{M \e^2 (1 - \gamma_R)^2}{ 4 (1 + \lambda)^2 \exp[ 2 B_R]}
\right].
\end{align*}
Using that $\lambda \le 1/2$ and taking the iterated expectation completes the proof.
\end{proof}

\begin{proof}[Proof of Theorem~\ref{theorem:mult_error_bound}]
A probability union bound implies
\begin{align*}
&\Prob\left( \left| \frac{1}{N} \sum_{m = 1}^{M} S_{\tau_C}^m(\phi)
- \int_\X \phi d\Pi \right| \ge \e \right)
\\
&\le \Prob\left( \left| \frac{1}{N} \sum_{m = 1}^M S_{\tau_C}^m(\phi)
- \int \K_C \phi d\Pi^N_Y \right| \ge \e/2 \right)
+ \Prob\left( \left| \int \K_C \phi d\Pi^N_Y - \int \phi d\Pi \right| \ge \e/2 \right).
\end{align*}
The proof is completed using Proposition~\ref{proposition:chernoff}, and Proposition~\ref{proposition:bias_bounded}.
\end{proof}

\section{Supporting technical results for Section~\ref{section:pg_application}}

\begin{proof}[Proof of Proposition~\ref{proposition:initial}]
By \eqref{eq:Hess_assumption} the largest eigenvalue of $\nabla^2 \l(x)$ uniformly over $x \in \R^d$ is bounded by $\lambda^*$.
Since $\l$ is twice continuously differentiable, using Taylor expansion, we have for all $x \in \R^d$
\begin{align*}
f(x) - f(x^*)
&= (x - x^* )^T \int_0^1 [  \nabla^2 \l (x^* + t (x - x^*)) + \Sigma^{-1} ] (x - x^* )  (1 - t) dt
\\
&\le \frac{1}{2} (x - x^* )^T [  \lambda^* + \Sigma^{-1} ] (x - x^* ).
\end{align*}
This then implies a lower bound on the normalizing constant using properties of the determinant
\begin{align*}
\int_{\R^d} \exp[-( f(x) - f(x^*) ] dx
&\ge \int_{\R^d} \exp\left[- \frac{1}{2} (x - x^* )^T [  \lambda^* + \Sigma^{-1} ] (x - x^* ) \right] dx
\\
&\ge (2 \pi)^{d/2} \det( \lambda^* I + \Sigma^{-1})^{-1/2}.
\end{align*}
By convexity of $\l$, Taylor expansion implies for each $x \in \R^d$
\[
f(x) - f(x^*)
\ge \frac{1}{2} (x - x^* )^T \Sigma^{-1} (x - x^* ).
\]
Combining these estimates and using properties of the determinant, we have the upper bound
\begin{align*}
&\int_{\R^d} \frac{\pi(x)}{q(x)} \pi(x) dx
\\
&= (2 \pi)^{d/2} (1/2 + h)^{d/2} \det(\Sigma)^{1/2} \frac{1}{Z^2}
\int_{\R^d} \exp\left[ \frac{1}{2 (1/2 + h)} (x - x^* )^T \Sigma^{-1} (x - x^* ) -2 f(x) \right]
dx
\\
&\le \frac{(1/2 + h)^{d/2}
\det( \lambda^* I + \Sigma^{-1}) \det(\Sigma)^{1/2} 
}{
(2 \pi)^{d/2}
}
\int_{\R^d}
\exp\left[ -\frac{1}{2} \frac{2 h}{1/2 + h} (x - x^* )^T \Sigma^{-1} (x - x^* ) \right]
dx
\\
&\le
\det( \lambda^* I + \Sigma^{-1}) \det(\Sigma) 
\frac{
(1/2 + h)^{d}
}{
(2h)^{d/2}
}
\\
&\le 
\det( \lambda^* \Sigma + I) 
\left( \frac{1}{ 2 \sqrt{2 h} } + \sqrt{\frac{h}{2}} \right)^{d}.
\end{align*}
\end{proof}

\begin{proof}[Proof of Proposition~\ref{proposition:pg_sampler}]
We look to apply Theorem~\ref{theorem:mse_bound}.
It is readily seen that the negative log-likelihood is infinitely differentiable, and for every $v \in \R^d$, 
\[
v^T \nabla^2 \l_n(\beta) v 
= v^T X^T \text{diag}[ \sigma(X \beta) (1 - \sigma(X \beta)) ] X v 
\le \frac{\norm{X^T X}_2}{4} \norm{v}_2^2.
\]
Since the Hessian is symmetric, we have shown that
\[
\norm{ \nabla^2 \l_n(\beta) }_2
\le \frac{\norm{X^T X}_2}{4}.
\]
We may then apply the bound in Proposition~\ref{proposition:initial} to the chosen importance sampling proposal $Q_h$.

The proof will be complete if we obtain a drift condition with $V$ defined by $V(\beta) = 1 + \norm{\beta}_2^2$.
Fix $\beta \in \R^d$ and we will take the conditional expectation of the first step of the Gibbs sampler with $\beta_1 | \w_1$ and then $\w_1 | \beta$. 
We have the variance identity
\begin{align*}
\E\left( \norm{\beta_1}_2^2 | \w_1, \beta \right)
&= \norm{\mu(\w_1)}_2^2 + \tr(\Sigma(\w_1)).
\end{align*}
Let $\Omega_1^{1/2} = \text{diag}(\w^{1/2})$.
Since $\Sigma$ is SPD, there exists a decomposition with a SPD matrix $\Sigma^{-1/2}$ so that the $\Sigma^{-1} = \Sigma^{-1/2} \Sigma^{-1/2}$.
By singular value decomposition \citep[Theorem 2.6.3]{Horn2012}, we may write
$\Omega_1^{1/2} X \Sigma^{1/2} = U_{\w_1} D_{\w_1} V_{\w_1}^T$ where $V_{\w_1} \in \R^{n \times n}$ is an orthogonal matrix so that $V_{\w_1}^T V_{\w_1} = V_{\w_1} V_{\w_1}^T = I_n$ and $U_{\w_1} \in \R^{d \times d}$ is also orthogonal and $D_{\w_1} \in \R^{n \times d}$ is a block diagonal matrix.
Then we have the decomposition
\begin{align*}
\Sigma(\w_1)
&= \Sigma^{1/2} ( [ \Omega_1^{1/2} X \Sigma^{1/2} ]^T \Omega_1^{1/2} X \Sigma^{1/2} + I )^{-1} \Sigma^{1/2}
\\
&= \Sigma^{1/2} V_{\w_1} ( D_{\w_1}^T D_{\w_1} + I )^{-1} V_{\w_1}^T \Sigma^{1/2}.
\end{align*}
This decomposiiton implies by the submultiplicative property of the matrix norm, \begin{align*}
\norm{\mu(\w_1)}_2^2
\le \norm{\Sigma(\w_1) }_2^2 \norm{ X^T (Y - (1/2) 1_d)}_2^2
\le \norm{ \Sigma^{1/2} }_2^2 \norm{ X^T (Y - (1/2) 1_d)}_2^2.
\end{align*}
Using cyclic property of the trace, we also have 
\begin{align*}
\tr\left[ \Sigma(\w_1) \right] 
\le \tr\left[ 
V_{\w_1}^T \Sigma V_{\w_1} ( D_{\w_1}^T D_{\w_1} + I )^{-1}
\right]
\le \tr\left[ 
V_{\w_1}^T \Sigma V_{\w_1}
\right]
\le \tr\left[ 
V_{\w_1} V_{\w_1}^T \Sigma
\right]
\le \tr\left[ \Sigma \right].
\end{align*}
Combining these estimates and taking the iterated expectation
\begin{align*}
\E\left( \norm{\beta_1}_2^2 + 1 | \beta \right)
&\le 1 + \norm{ \Sigma }_2^2 \norm{ X^T (Y - 1/2 1_d) }_2^2 + \tr(\Sigma).
\end{align*}.

This drift condition combined with a minorization shows the posterior is the unique invariant measure \citep{hairer:2021} or Harris recurrence would work as well \citep[Theorem 10.4.4]{meyn:tweedie:2009}.
\end{proof}

\bibliography{main.bib}

\end{document}